\documentclass[10pt]{amsart}
\usepackage{a4wide,color}
\usepackage{amsmath,amssymb}
\usepackage[active]{srcltx}
\usepackage{ulem}

\allowdisplaybreaks

\let \eps \varepsilon

\newtheorem{theorem}{Theorem}
\newtheorem{lemma}[theorem]{Lemma}
\newtheorem{proposition}[theorem]{Proposition}
\newtheorem{remark}[theorem]{Remark}
\newtheorem{corollary}[theorem]{Corollary}

\let \eps=\varepsilon

\let \la=\lambda

\let \ka=\kappa

\let \om=\omega

\begin{document}
\title[Boltzmann equation]{Stability of background perturbation for
Boltzmann equation}
\author{Yu-Chu Lin}
\address{Yu-Chu Lin, Department of Mathematics, National Cheng Kung
University, Tainan, Taiwan}
\email{yuchu@mail.ncku.edu.tw }
\author{Haitao Wang}
\address{Haitao Wang, School of Mathematical Sciences, Institute of Natural
Sciences, MOE-LSC, IMA-Shanghai, Shanghai Jiao Tong University, Shanghai,
China}
\email{haitallica@sjtu.edu.cn}
\author{Kung-Chien Wu}
\address{Kung-Chien Wu, Department of Mathematics, National Cheng Kung
University, Tainan, Taiwan and National Center for Theoretical Sciences,
National Taiwan University, Taipei, Taiwan}
\email{kungchienwu@gmail.com}
\date{\today }
\thanks{2020 Mathematics Subject Classification: 35Q20; 82C40.}

\begin{abstract}
Consider the Boltzmann equation in the perturbation regime. Since the
macroscopic quantities in the background global Maxwellian are obtained
through measurements, there are typically some errors involved. This paper
investigates the effect of background variations on the solution for a given
initial perturbation. Our findings demonstrate that the solution changes
continuously with variations in the background and provide a sharp time decay estimate of the associated errors. The proof relies on refined estimates
for the linearized solution operator and a proper decomposition of the
nonlinear solution. %
%
\end{abstract}

\keywords{Boltzmann equation, Maxwellian states, Stability.}
\maketitle






\section{Introduction}

\subsection{The model}

The Boltzmann equation reads
\begin{equation}
\left \{
\begin{array}{l}
\partial _{t}F+\xi \cdot \nabla _{x}F=Q(F,F)\text{,} \\[4mm]
F(0,x,\xi )=F_{0}(x,\xi )\text{,}%
\end{array}%
\right. \quad (t,x,\xi )\in {\mathbb{R}}^{+}\times {\mathbb{R}}^{3}\times {%
\mathbb{R}}^{3}\text{,}  \label{bot.1.a}
\end{equation}%
where $F(t,x,\xi )$ is the velocity distribution function for the particles
at time $t>0$, position $x=(x_{1},x_{2},x_{3})\in {\mathbb{R}}^{3}$ and
microscopic velocity $\xi =(\xi _{1},\xi _{2},\xi _{3})\in {\mathbb{R}}^{3}$%
. The left-hand side of this equation models the transport of particles and
the operator on the right-hand side models the effect of collisions during
the transport,
\begin{equation*}
Q(F,G)=\frac{1}{2}\int_{{\mathbb{R}}^{3}\times S^{2}}|\xi -\xi _{\ast }|^{\gamma
}B(\vartheta )\left \{ F_{\ast }^{^{\prime }}G^{\prime }+G_{\ast }^{\prime
}F^{\prime }-F_{\ast }G-FG_{\ast }\right \} d\xi _{\ast }d\om \text{.}
\end{equation*}%
Here the usual convention, i.e., $F=F\left( t,x,\xi \right) $, $F_{\ast
}\left( t,x,\xi _{\ast }\right) $, $F^{\prime }=F(t,x,\xi )$ and $F_{\ast
}^{\prime }=F\left( t,x,\xi _{\ast }^{\prime }\right) $, is used; the
post-collisional velocities of particles satisfy
\begin{equation*}
\xi ^{\prime }=\xi -[(\xi -\xi _{\ast })\cdot \om]\om \text{,}\quad \xi
_{\ast }^{\prime }=\xi_{*} +[(\xi -\xi _{\ast })\cdot \om]\om \text{, }\om \in
S^{2}\text{.}
\end{equation*}%
Throughout this paper, we consider the Maxwellian molecules $\left( \gamma
=0\right) $ and hard potentials ($0<\gamma \leq 1$); and $B(\vartheta )$
satisfies the Grad cutoff assumption
\begin{equation*}
0<B(\vartheta )\leq C\left \vert \cos \vartheta \right \vert \text{,}
\end{equation*}%
for some constant $C>0$, and $\vartheta $ is defined by
\begin{equation*}
\cos \vartheta =\frac{|(\xi -\xi _{\ast })\cdot \om|}{|\xi -\xi _{\ast }|}%
\text{.}
\end{equation*}%
The global Maxwellian states $\mathcal{M}_{[\rho ,\mu ,T]}$, with $[\rho
,\mu ,T]$ constant,
\begin{equation*}
\mathcal{M}_{[\rho ,\mu ,T]}=\frac{\rho }{\left( 2\pi RT\right) ^{3/2}}e^{-%
\frac{\left \vert \xi -\mu \right \vert ^{2}}{2RT}}
\end{equation*}%
satisfy $Q(\mathcal{M}_{[\rho ,\mu ,T]},\mathcal{M}_{[\rho ,\mu ,T]})=0$ and
are steady solutions of the Boltzmann equation.

As is well-known, in the perturbative framework, the evolution of the
initial perturbation crucially depends on the background Maxwellian. For
example, transport coefficients derived from the Boltzmann equation, such as
viscosity, heat conductivity, and macroscopic Euler waves, are all
determined by the background Maxwellian (see \cite%
{[dilute],[LiuYu2],[Sone],[Uaki-Yang-2]} for more information).

Typically, macroscopic quantities in the global Maxwellian, such as density,
velocity, and temperature, are obtained through measurements, which can
involve errors. Therefore, it is natural to ask the following questions:

\begin{enumerate}
\item How does the solution change as the background varies for the same
initial perturbation?

\item Can we obtain a sharp estimate for the difference between solutions
associated with different background Maxwellians?
\end{enumerate}

In this paper, we aim to answer the aforementioned questions by studying the
stability of the solution of (\ref{bot.1.a}) under perturbations of the
background with respect to macroscopic quantities. Let us now formulate our
problem. Let $F^{a}$ and $F^{b}$ be solutions to the Boltzmann equation with
the same initial perturbation but for different global Maxwellians. That is,
\begin{equation*}
\left \{
\begin{array}{l}
\partial _{t}F^{a}+\xi \cdot \nabla _{x}F^{a}=Q\left( F^{a},F^{a}\right)
\text{,}%
\vspace {3mm}
\\
F^{a}\left( 0,x,\xi \right) =F_{in}^{a}\left( x,\xi \right) =\mathcal{M}%
_{a}\left( \xi \right) +F_{0}\text{,}%
\end{array}%
\right.
\end{equation*}%
and%
\begin{equation*}
\left \{
\begin{array}{l}
\partial _{t}F^{b}+\xi \cdot \nabla _{x}F^{b}=Q\left( F^{b},F^{b}\right)
\text{,}%
\vspace {3mm}
\\
F^{b}\left( 0,x,\xi \right) =F_{in}^{b}\left( x,\xi \right) =\mathcal{M}%
_{b}\left( \xi \right) +F_{0}\text{,}%
\end{array}%
\right.
\end{equation*}%
respectively, where $\mathcal{M}_{a}\left( \xi \right) $ and $\mathcal{M}%
_{b}\left( \xi \right) $ are two global Maxwellians
\begin{equation*}
\mathcal{M}_{a}\left( \xi \right) =\frac{\rho _{a}}{\left( 2\pi
RT_{a}\right) ^{3/2}}e^{-\frac{\left \vert \xi -\mu _{a}\right \vert ^{2}}{%
2RT_{a}}}\text{,}\quad \mathcal{M}_{b}\left( \xi \right) =\frac{\rho _{b}}{%
\left( 2\pi RT_{b}\right) ^{3/2}}e^{-\frac{\left \vert \xi -\mu
_{b}\right \vert ^{2}}{2RT_{b}}}\text{,}
\end{equation*}%
with $\lvert \rho _{a}-\rho _{b}\rvert +\left \vert \mu _{a}-\mu
_{b}\right \vert +\left \vert T_{a}-T_{b}\right \vert >0$.

In the perturbation regime, if we let $F^{a}=\mathcal{M}_{a}+\sqrt{\mathcal{M%
}_{a}}f^{a}$ and $F^{b}=\mathcal{M}_{b}+\sqrt{\mathcal{M}_{b}}f^{b}$, we can
obtain that the perturbation functions $f^{\alpha }$, where $\alpha =a$, $b$%
, satisfy the following equations:
\begin{equation}
\left \{
\begin{array}{l}
\partial _{t}f^{\alpha }+\xi \cdot \nabla _{x}f^{\alpha }=\mathcal{L}%
_{\alpha }f^{\alpha }+\Gamma _{\alpha }\left( f^{\alpha },f^{\alpha }\right)
\vspace {3mm}
\\
f^{\alpha }\left( 0,x,\xi \right) =\varepsilon f_{0}^{\alpha }\left( x,\xi
\right) \text{,}%
\end{array}%
\right.   \label{nonlinear}
\end{equation}%
where%
\begin{equation*}
\mathcal{L}_{\alpha }h=\frac{2}{\sqrt{\mathcal{M}_{\alpha }}} Q\left(
\mathcal{M}_{\alpha },\sqrt{\mathcal{M}_{\alpha }}h\right)  \text{,}%
\quad \Gamma _{\alpha }\left( h_{1},h_{2}\right) =\frac{1}{\sqrt{\mathcal{M}%
_{\alpha }}}Q\left( \sqrt{\mathcal{M}_{\alpha }}h_{1},\sqrt{\mathcal{M}%
_{\alpha }}h_{2}\right) \text{.}
\end{equation*}%
As the initial perturbations are the same, one has
\begin{equation}
\sqrt{\mathcal{M}_{a}}f_{0}^{a}=\sqrt{\mathcal{M}_{b}}f_{0}^{b}\text{.}
\label{init}
\end{equation}

Moreover, let $G=F^{b}-F^{a}$ and then it satisfies%
\begin{equation}
\left \{
\begin{array}{l}
\partial _{t}G+\xi \cdot \nabla _{x}G=Q\left( G,G\right) +2Q\left(
F^{a},G\right) \text{,}%
\vspace {3mm}
\\
G\left( 0,x,\xi \right) =\mathcal{M}_{b}\left( \xi \right) -\mathcal{M}%
_{a}\left( \xi \right) \text{.}%
\end{array}%
\right.   \label{bot-diff}
\end{equation}%
Fixing $\mathcal{M}_{a}$ as a reference background, one can consider $G=%
\mathcal{M}_{b}-\mathcal{M}_{a}+\sqrt{\mathcal{M}_{a}}g$ and so
\begin{equation}
g=\frac{\left( F^{b}-\mathcal{M}_{b}\right) -\left( F^{a}-\mathcal{M}%
_{a}\right) }{\sqrt{\mathcal{M}_{a}}}\text{.}  \label{g-rep}
\end{equation}%
Plugging this into (\ref{bot-diff}) gives%
\begin{equation}
\left \{ \begin{aligned} & \partial _{t}g+\xi \cdot \nabla
_{x}g=\mathcal{L}_{a}g+2\Gamma _{a}\left(
\frac{\sqrt{\mathcal{M}_{b}}}{\sqrt{\mathcal{M}_{a}}}f^{b},\frac{%
\mathcal{M}_{b}-\mathcal{M}_{a}}{\sqrt{\mathcal{M}_{a}}}\right) +2\Gamma
_{a}\left( f^{a},g\right) +\Gamma _{a}\left( g,g\right) \text{,} \\ &g\left(
0,x,\xi \right) =0. \end{aligned}\right.   \label{g-eqn1}
\end{equation}

Noting that $\sqrt{\mathcal{M}_{a}}g$ represents the difference between the
perturbation solutions, and the objective of this paper is to investigate
the quantitative behavior of $g$. Since the Boltzmann equation is invariant
under Galilean transforms, we may assume that one of the Maxwellians has
zero macroscopic velocity. Then, without loss of generality and through
appropriate normalization, the global Maxwellians $\mathcal{M}_{a}$ and $%
\mathcal{M}_{b}$ can be assumed as
\begin{equation}
\mathcal{M}_{a}=\mathcal{M}=\frac{1}{\left( 2\pi \right) ^{3/2}}\exp \left( -%
\frac{\left \vert \xi \right \vert ^{2}}{2}\right) \text{,}\; \mathcal{M}_{b}=%
\frac{\rho }{\left( 2\pi \lambda \right) ^{3/2}}\exp \left( -\frac{%
\left \vert \xi -\mu \right \vert ^{2}}{2\lambda }\right) \text{.}
\label{Maxwellian-assu1}
\end{equation}%
%
%
%
%
%
For simplicity, we hereafter denote $\mathcal{L}_{a}$ and $\Gamma _{a}$ by $%
\mathcal{L}$ and $\Gamma $.

There are extensive studies on the stability of Boltzmann equation in the
literature, with initial data either near the vacuum or near the same global
Maxwellian, see \cite{[Ha],[HaXiao]} and references therein. It is important
to note that our problem in equation \eqref{g-eqn1} presents a significantly
different setting than the prior research.

\subsection{Notations}

Let us define some notations used in this paper. We denote $\left \langle
\xi \right \rangle ^{s}=(1+|\xi |^{2})^{s/2}$, $s\in {\mathbb{R}}$. For the
microscopic variable $\xi $, we denote the Lebesgue spaces
\begin{equation*}
|g|_{L_{\xi }^{q}}=\Big(\int_{{\mathbb{R}}^{3}}|g|^{q}d\xi \Big)^{1/q}\text{
if }1\leq q<\infty \text{,}\quad \quad |g|_{L_{\xi }^{\infty }}=\sup_{\xi
\in {\mathbb{R}}^{3}}|g(\xi )|\text{,}
\end{equation*}%
and the weighted norms can be defined by
\begin{equation*}
|g|_{L_{\xi ,\beta }^{q}}=\Big(\int_{{\mathbb{R}}^{3}}\left \vert \left
\langle \xi \right \rangle ^{\beta }g\right \vert ^{q}d\xi \Big)^{1/q}\text{
if }1\leq q<\infty \text{,}\quad \quad |g|_{L_{\xi ,\beta }^{\infty
}}=\sup_{\xi \in {\mathbb{R}}^{3}}\left \vert \left \langle \xi \right
\rangle ^{\beta }g(\xi )\right \vert \text{,}
\end{equation*}%
and
\begin{equation*}
|g|_{L_{\xi }^{\infty }(m)}=\sup_{\xi \in {\mathbb{R}}^{3}}\left \{ |g(\xi
)|m(\xi)\right \} \text{,}
\end{equation*}%
where $\beta \in {\mathbb{R}}$ and $m$ is a weight function. The $L_{\xi
}^{2}$ inner product in ${\mathbb{R}}^{3}$ will be denoted by $\big<\cdot
,\cdot \big>_{\xi }$, i.e.,
\begin{equation*}
\left \langle f,g\right \rangle _{\xi }=\int f(\xi )\overline{g(\xi )}d\xi
\text{.}
\end{equation*}%
For the Boltzmann equation with cut-off potential, the natural norm in $\xi $
is $|\cdot |_{L_{\sigma }^{2}}$, which is defined as
\begin{equation*}
|g|_{L_{\sigma }^{2}}^{2}=\left|\left \langle \xi \right \rangle ^{\frac{%
\gamma }{2}}g\right|_{L_{\xi }^{2}}^{2}\text{.}
\end{equation*}%
For the space variable $x$, we have similar notations, namely,
\begin{equation*}
|g|_{L_{x}^{q}}=\left( \int_{{\mathbb{R}^{3}}}|g|^{q}dx\right) ^{1/q}\text{
if }1\leq q<\infty \text{,}\quad \quad |g|_{L_{x}^{\infty }}=\sup_{x\in {%
\mathbb{R}^{3}}}|g(x)|\text{.}
\end{equation*}%
Furthermore, we define the high order Sobolev norm: let $s\in {\mathbb{N}}$
and define
\begin{equation*}
\left \vert g\right \vert _{H_{\xi }^{s}}=\sum_{|\alpha |\leq s}\left \vert
\partial _{\xi }^{\alpha }g\right \vert _{L_{\xi }^{2}}\text{,\  \  \  \  \  \ }%
\left \vert g\right \vert _{H_{x}^{s}}=\sum_{|\alpha |\leq s}\left \vert
\partial _{x}^{\alpha }g\right \vert _{L_{x}^{2}}\text{,}
\end{equation*}%
where $\alpha $ is any multi-index with $|\alpha |\leq s$.

Finally, with $\mathcal{X}$ and $\mathcal{Y}$ being norm spaces, we define
\begin{equation*}
\left \Vert g\right \Vert _{\mathcal{XY}}=\left \vert \left \vert g\right
\vert _{\mathcal{Y}}\right \vert _{\mathcal{X}}\text{.}
\end{equation*}%
We also denote
\begin{equation*}
\Vert g\Vert _{L^{2}}=\Vert g\Vert _{L_{\xi }^{2}L_{x}^{2}}=\left( \int_{{%
\mathbb{R}^{3}}}|g|_{L_{x}^{2}}^{2}d\xi \right) ^{1/2}\text{.}
\end{equation*}


For simplicity of notations, hereafter, we abbreviate \textquotedblleft {\ $%
\leq C$} \textquotedblright \ to \textquotedblleft {\ $\lesssim $ }%
\textquotedblright , where $C$ is a positive constant depending only on
fixed numbers.

\subsection{Heuristics and a toy model}

Starting from \eqref{g-rep}, we can express $g$ as
\begin{equation*}
\frac{\sqrt{\mathcal{M}_{b}}}{\sqrt{\mathcal{M}_{a}}}f^{b}-f^{a}\text{,}
\end{equation*}%
where $f^{a}$ and $f^{b}$ are the solutions to \eqref{nonlinear}. Let $%
\mathbb{G}_{\alpha }^{t}$ denote the semi-group generated by the linearized
operator $-\xi \cdot \nabla _{x}+\mathcal{L}_{\alpha }$ for $\alpha =a$, $b$%
. Since we are considering the perturbative regime, it is reasonable to
expect that the behavior of $f^{\alpha }$ is mainly governed by the
linearized equation, i.e., $f^{\alpha }(t)\approx \mathbb{G}_{\alpha
}^{t}f^{\alpha }(0)$, this is because the decay of the nonlinear part is
faster than the linear part, and thus
\begin{equation*}
g\approx \frac{\sqrt{\mathcal{M}_{b}}}{\sqrt{\mathcal{M}_{a}}}\mathbb{G}%
_{b}^{t}f^{b}(0)-\mathbb{G}_{a}^{t}f^{a}(0)=\left( \frac{\sqrt{\mathcal{M}%
_{b}}}{\sqrt{\mathcal{M}_{a}}}\mathbb{G}_{b}^{t}\frac{\sqrt{\mathcal{M}_{a}}%
}{\sqrt{\mathcal{M}_{b}}}-\mathbb{G}_{a}^{t}\right) \varepsilon f_{0}^{a}%
\text{,}
\end{equation*}%
where the last equality is due to \eqref{init}.

This observation suggests to us that the behavior of $g$ should be expressed
in terms of the difference between the two solution operators for linearized
Boltzmann equations. Moreover, for the linearized Boltzmann equation, by the
Chapman-Enskog expansion, it is known that the large-time behavior of the
solution to the linearized equation is governed by the linear Navier-Stokes
equation, where the viscosity and heat conductivity are proportional to the
power (precisely, the power is $\frac{2-\gamma }{2}$) of the macroscopic
temperature (see \cite{[EP], [Kawashima-Matsumura-Nishida]} for more
details).

Inspired by the linear Navier-Stokes equation, we consider the possibly
simplest toy model, the heat equation. Given two sets of heat equations,
\begin{equation}
\left \{
\begin{array}{l}
\partial _{t}h^{\alpha }+\mu _{\alpha }\cdot \nabla _{x}h^{\alpha }=\lambda
^{\left( 2-\gamma \right) /2}\Delta h^{\alpha }%
\vspace {3mm}
\\
\left. h^{\alpha }\right \vert _{t=0}=h_{0}\text{, }\quad \alpha =a\text{, }b%
\text{,}%
\end{array}%
\right.   \label{eq:heat}
\end{equation}%
here, positive constants $\mu _{\alpha }$ and $\lambda _{\alpha }$ are used
to mimic the macroscopic velocity and temperature, respectively. Using
Galilean transformation and suitable scaling, we assume that
\begin{equation*}
\mu _{a}=0\text{,}\; \mu _{b}=\mu \text{,}\; \lambda _{a}=1\text{, }\lambda
_{b}=\lambda \text{.}
\end{equation*}%
To analyze the difference between $h^{a}$ and $h^{b}$, we can use an
explicit heat kernel representation to yield
\begin{equation*}
h^{b}(t,x)-h^{a}(t,x)=\int_{\mathbb{R}^{3}}\left[ \frac{1}{(4\pi \kappa
t)^{3/2}}e^{-\frac{\left \vert x-\mu t-y\right \vert ^{2}}{4\kappa t}}-\frac{1%
}{\left( 4\pi t\right) ^{3/2}}e^{-\frac{\left \vert x-y\right \vert ^{2}}{4t}}%
\right] h_{0}(y)dy\text{,}\quad \kappa =\lambda ^{\left( 2-\gamma \right) /2}%
\text{.}
\end{equation*}%
By calculating the difference between two heat kernels explicitly (see
Appendix \ref{sec:heat}), one obtains the sharp estimates for $h^{b}-h^{a}$:
\begin{equation}
\left \{ \begin{aligned} & \left \vert h^{b}-h^{a}\right \vert
_{L_{x}^{\infty }}\leq C(1+t)^{-1}\left \vert h_{0}\right \vert
_{L_{x}^{1}}\left( \left \vert \lambda -1\right \vert (1+t)^{-1/2}+\left
\vert \mu \right \vert \right) \text{,}\\ &\left \vert h^{b}-h^{a}\right
\vert _{L_{x}^{2}}\leq C(1+t)^{-1/4}\left \vert h_{0}\right \vert
_{L_{x}^{1}}\left( \left \vert \lambda -1\right \vert (1+t)^{-1/2}+\left
\vert \mu \right \vert \right) \text{,} \end{aligned}\right.
\label{eq:est-toy}
\end{equation}%
for $t\geq 1$.

It should be noted that this method heavily relies on the explicit
expressions of heat kernels. For Boltzmann equation, while there exist
constructions of Green's functions (see \cite{[LiuYu],[LiuYu1],[LiuYu2]}),
the expressions are not precise enough to obtain sharp estimates for the
difference between them.

Alternatively, for the toy model \eqref{eq:heat}, we set $h=h^{a}-h^{b}$ to
obtain
\begin{equation}
\left \{
\begin{array}{l}
\partial _{t}h=\Delta h+\mu \cdot \nabla _{x}h^{b}-\left( \lambda
^{\left( 2-\gamma \right) /2}-1\right) \Delta h^{b}%
\vspace {3mm}
\\
h\left( 0,x\right) =0\text{.}%
\end{array}%
\right.  \label{eq:heat-2}
\end{equation}%
One could still recover \eqref{eq:est-toy} by refining estimates for the
heat kernel and $h^{b}$. See Appendix \ref{sec:heat} for details.

Compared with \eqref{g-eqn1}, one may view \eqref{eq:heat-2} as a simplified
analogue of it. Therefore, it is natural to ask whether we can establish
similar estimates as \eqref{eq:est-toy} for the difference $g$. Our main
result provides an affirmative answer to this question.

\subsection{Main theorem and idea of proof}

In our main theorem, we assume the initial condition $f_{0}^{a}$ satisfies
$$(1+\left \vert \xi \right \vert ^{2})^{\beta /2}f_{0}^{a }\in L_{\xi
}^{\infty }\left( L_{x}^{1}\cap L_{x}^{\infty }\right)\,, \quad \beta >3/2+2\gamma\,,$$
 in order to ensure the existence of the solution of the Boltzmann equation $f^{a}$ in the space $L^{\infty}_{\xi,\beta}(L^{\infty}_{x}\cap L^{2}_{x})$ and to control the nonlinear part of $g$. Moreover,
in view of the assumption $\left( \ref{init}%
\right) $, when $(1+\left \vert \xi \right \vert ^{2})^{\beta /2}f_{0}^{a}\in
L_{\xi }^{\infty }\left( L_{x}^{1}\cap L_{x}^{\infty }\right) $ is assumed, $%
f_{0}^{b}$  cannot be arbitrary and is determined by
\begin{equation*}
(1+\left \vert \xi \right \vert ^{2})^{\beta /2}\frac{\sqrt{\mathcal{M}_{b}}}{%
\sqrt{\mathcal{M}_{a}}}f_{0}^{b}=(1+\left \vert \xi \right \vert ^{2})^{\beta
/2}f_{0}^{a}\in L_{\xi }^{\infty }\left( L_{x}^{1}\cap L_{x}^{\infty
}\right)
\end{equation*}%
where%
\begin{align*}
\frac{\sqrt{\mathcal{M}_{b}}}{\sqrt{\mathcal{M}_{a}}}&=\frac{\rho ^{1/2}}{%
\lambda ^{3/4}}\exp \left \{ \frac{|\xi -\mu |^{2}}{4}\left( 1-\frac{1}{%
\lambda }\right) +\frac{2(\xi -\mu )\cdot \mu +|\mu |^{2}}{4}\right \} \\
&=\frac{\rho ^{1/2}}{%
\lambda ^{3/4}}\exp \left \{ \frac{\la-1}{4\la}\left|\xi +\frac{\mu}{\la-1} \right|^{2} +\frac{|\mu|^{2}}{4(\la-1)}\right \} \text{,
}
\end{align*}%
hence we need to assume $\lambda >1$ later on. The assumption $\lambda<2$ comes from the estimate of the difference of two Maxwellians in Lemma \ref{Maxwellian-Esti}.

The main theorem of this paper is stated as follows:
\begin{theorem}
Let $1<\underline{\la}<\lambda <\overline{\lambda }<2$, $0<\rho <\overline{\rho }$,  $0<|\mu| <|\overline{\mu }|$.
Assume that $f^{a}$ and $f^{b}$ are solutions to (\ref{nonlinear})
corresponding to $\mathcal{M}_{\alpha }=\mathcal{M}_{a}$ and $\mathcal{M}%
_{\alpha }=\mathcal{M}_{b}$ defined as (\ref{Maxwellian-assu1}), with the
initial data $\left \langle \xi \right \rangle ^{\beta }f_{0}^{a}\in \left(
L_{x}^{1}\cap L_{x}^{\infty }\right) L_{\xi }^{\infty }$, $\beta
>3/2+2\gamma $, and $\varepsilon >0$ small.
Then
\begin{equation*}
\left \Vert \frac{\left( F^{b}-\mathcal{M}_{b}\right) -\left( F^{a}-\mathcal{M%
}_{a}\right) }{\sqrt{\mathcal{M}_{a}}}\right \Vert _{L_{\xi ,\beta }^{\infty
}L_{x}^{\infty }}\leq C\varepsilon (1+t)^{-1}\Bigl(\lvert \rho -1\rvert
+\left \vert \mu \right \vert +\left \vert \lambda -1\right \vert \Bigr)\text{,}
\end{equation*}%
and
\begin{equation*}
\left \Vert \frac{\left( F^{b}-\mathcal{M}_{b}\right) -\left( F^{a}-\mathcal{M%
}_{a}\right) }{\sqrt{\mathcal{M}_{a}}}\right \Vert _{L_{\xi ,\beta }^{\infty
}L_{x}^{2}}\leq C\varepsilon (1+t)^{-1/4}\Bigl(\lvert \rho -1\rvert
+\left \vert \mu \right \vert +\left \vert \lambda -1\right \vert \Bigr)\text{.}
\end{equation*}%
for some constant $C>0$ dependent of $\overline{\rho}$, $\underline{\la}$, $\overline{\lambda }$ and $|\overline{\mu}|$, but independent of time $t$.
\end{theorem}

Some comments on the theorem are as follows:

\begin{remark}
We say the estimate is sharp in the sense that the decay rate of the
estimate is $(1+t)^{-1}$, which is the same as that of the toy model %
\eqref{eq:est-toy}. However, for the toy model, the decay rate for speed
variation is $(1+t)^{-1}$, while for variation of the diffusion coefficient
it is $(1+t)^{-3/2}$. In contrast, for the Boltzmann equation, we can only
obtain $(1+t)^{-1}$ for variations of all macroscopic quantities. So far, it
is unknown whether it is possible to obtain different decay rates for
different quantities.
\end{remark}

\begin{remark}
The result is established using the $L^\infty_x$ framework, and no Sobolev
regularity is required. Moreover, the error estimate is given in terms of
the $L^\infty_x$ norm, which appears to be more realistic in terms of
measurement.
\end{remark}

\begin{remark}
The theorem requires the strict inequality $1<\underline{\la}<\lambda <\overline{\lambda }<2$%
. However, it can also be proven for the case $\lambda =1$, provided that we
assume suitable exponential velocity weight on the initial data.
\end{remark}

We will now outline the main idea and strategy to prove our main result. For
the equation \eqref{g-eqn1}, we use the solution operator $\mathbb{G}^{t}$
for the linearized operator $-\xi \cdot \nabla _{x}+\mathcal{L}$ and
Duhamel's principle to represent $g$ as follows:
\begin{equation}
\begin{aligned} g(t) & = 2\int_{0}^{t} \mathbb{G}^{t-\tau}\Gamma _{a}\left(
\frac{\sqrt{\mathcal{M}_{b}}}{\sqrt{\mathcal{M}_{a}}}f^{b},\frac{\mathcal{M}
_{b}-\mathcal{M}_{a}}{\sqrt{\mathcal{M}_{a}}}\right)(\tau) d\tau \\ &\quad +
2\int_{0}^{t} \mathbb{G}^{t-\tau}\Gamma _{a}\left( f^{a},g\right)(\tau)d\tau +
\int_{0}^{t} \mathbb{G}^{t-\tau}\Gamma _{a}\left( g,g\right)(\tau) d\tau. \end{aligned}
\label{eq:duhamel}
\end{equation}%
Since the first term on the right hand side is expected to dominate the
behavior of $g$, our primary objective is to obtain a sharp estimate for it.
However, the decay estimates available for the linearized equation (as shown
in Theorem \ref{Thm for linear}) and $f^{b}$ (as seen in Theorem \ref{prop:
nonlinear}) can only yield a $(1+t)^{-3/4}$ decay rate for the $%
L_{x}^{\infty }$ norm (as indicated in Remark \ref{rmk:slow}). The challenge of this paper is to improve the decay of this part from $(1+t)^{-3/4}$ to $(1+t)^{-1}$. Therefore,
we need more precise estimates for both $\mathbb{G}^{t}$ and $f^{b}$.
Through spectral analysis, we decompose the operator $\mathbb{G}^{t}$ into
three parts:
\begin{equation*}
\mathbb{G}^{t}=\mathbb{G}_{L;0}^{t}+\mathbb{G}_{L;\perp }^{t}+\mathbb{G}%
_{S}^{t}\text{,}
\end{equation*}%
where $\mathbb{G}_{L;0}^{t}$ is the long wave fluid part, $\mathbb{G}%
_{L;\perp }^{t}$ is the long wave non-fluid part, and $\mathbb{G}_{S}^{t}$
is the short wave part. These three parts have different natures. Among
them, only $\mathbb{G}_{L;0}^{t}$ behaves like the heat kernel in the toy
model \eqref{eq:heat}. We exploit its space-time pointwise structure to
obtain $L_{x}^{p}-L_{x}^{q}$-type estimates. However, for $\mathbb{G}%
_{L;\perp }^{t}$ and $\mathbb{G}_{S}^{t}$, due to insufficient spectral
information, we need to work in the $\lVert \cdot \rVert _{L_{x}^{2}L_{\xi
}^{2}}$ setting, as shown in Proposition \ref{Prop-LS}.

Next, motivated by the Liu-Yu's Green's function approach \cite{[LiuYu]} and
bootstrap argument, we improve the estimate from $L_{x}^{2}L_{\xi }^{2}$ to $L_{\xi
,\beta }^{\infty }L_{x}^{r}$ and obtain a more precise estimate for the
semi-group than the classical results \cite%
{[dilute],[Uaki-Yang],[Uaki-Yang-2]}. Moreover, we apply the semi-group
estimate to the long and short wave parts of the linearized solution to
obtain decay rates, which reveal their different structures. These decay
estimates for the semi-group $\mathbb{G}^{t}$ (including Propositions \ref%
{Prop-LS} and \ref{Prop-linear}, Corollaries \ref{lemm-long} and \ref%
{lemm-short}) are themselves new.

However, even with the above refined estimates for the semi-group, the slow
decay of $f^{b}$ still poses a challenge when applying it to the first term
on the right hand side of \eqref{eq:duhamel}. To address this issue, we
decompose $f^{b}$ into linear and nonlinear parts:
\begin{equation*}
f^{b}=\varepsilon \mathbb{G}_{b}^{t}f_{0}^{b}+\int_{0}^{t}\mathbb{G}%
_{b}^{t-\tau }\Gamma _{b}\left( f^{b},f^{b}\right) d\tau \text{.}
\end{equation*}%
Here $\mathbb{G}_{b}^{t}$ is the semi-group generated by $-\xi \cdot \nabla
_{x}+\mathcal{L}_{b}$. The nonlinear part decays fast and is not
problematic. We decompose linear part further into long wave part $\mathbb{G}%
_{b,L}^{t}f_{0}^{b}$ and short wave parts $\mathbb{G}_{b,S}^{t}f_{0}^{b}$.
We carefully analyze each term in the following integral
\begin{equation*}
\int_{0}^{t}\Bigl(\mathbb{G}_{L;0}^{t-\tau }+\mathbb{G}_{L;\perp }^{t-\tau }+%
\mathbb{G}_{S}^{t-\tau }\Bigr)\Gamma _{a}\left( \frac{\sqrt{\mathcal{M}_{b}}%
}{\sqrt{\mathcal{M}_{a}}}\Bigl(\mathbb{G}_{b;L}^{\tau }f_{0}^{b}+\mathbb{G}%
_{b;S}^{\tau }f_{0}^{b}\Bigr),\frac{\mathcal{M}_{b}-\mathcal{M}_{a}}{\sqrt{%
\mathcal{M}_{a}}}\right) d\tau \text{.}
\end{equation*}%
By doing so, we eventually recover the sharp decay estimate. Notably, only
the term
\begin{equation*}
\mathbb{G}_{L;0}^{t-\tau }(\cdots )\mathbb{G}_{b;L}^{\tau }
\end{equation*}%
behaves similarly to the counterpart in the toy model. Once we have obtained
the sharp estimate, we propose an appropriate ansatz for $g$ and close the
nonlinear problem by an a priori estimate.

\subsection{Organization of the paper}

The rest of this paper is structured as follows: In Section \ref{pre}, we
begin by introducing some basic properties of the operators $\mathcal{L}$
and $\Gamma$. We then provide a review of previously established results
concerning decay estimates of semi-groups and perturbation problems. After
that, we develop new refined estimates for the semi-group. Section \ref%
{sec:proof-main} is dedicated to the proof of the main theorem, while
Appendix \ref{sec:heat} contains detailed estimates for the toy model.

\section{Results for the problem around $\mathcal{M}_{a}$}

\label{pre}

\subsection{Basic estimates for $\mathcal{L}$, $\Gamma$ and well-posedness
results}

\label{collision} It is well known that the null space of $\mathcal{L}$ is a
five-dimensional vector space with the orthonormal basis $\{ \chi
_{i}\}_{i=0}^{4}$, where
\begin{equation*}
\mathrm{Ker}(\mathcal{L})=\left \{ \chi _{0},\chi _{i},\chi _{4}\right \}
=\left \{ \mathcal{M}^{1/2},\  \xi _{i}\mathcal{M}^{1/2},\  \frac{1}{\sqrt{6}}%
(|\xi |^{2}-3)\mathcal{M}^{1/2},~i=1,2,3\right \} \text{.}
\end{equation*}%
Based on this property, we can introduce the macro-micro decomposition: let $%
\mathrm{P}_{0}$ be the orthogonal projection with respect to the $L_{\xi
}^{2}$ inner product onto $\mathrm{Ker}(\mathcal{L})$, and $\mathrm{P}%
_{1}\equiv \mathrm{Id}-\mathrm{P}_{0}$.

The collision operator $\mathcal{L}$ consists of a multiplicative operator $%
\nu (\xi )$ and an integral operator $K$:
\begin{equation*}
\mathcal{L}f=-\nu (\xi )f+Kf\, \text{,}
\end{equation*}%
where
\begin{equation*}
\nu (\xi )=\int_{\mathbb{R}^{3}\times S^{2}}B(\vartheta )|\xi -\xi _{\ast
}|^{\gamma }\mathcal{M}(\xi _{\ast })d\xi _{\ast }d\om \text{,}
\end{equation*}%
and
\begin{equation*}
Kf=-K_{1}f+K_{2}f  \label{defK}
\end{equation*}%
is defined as \cite{[Glassey], [Grad]}:
\begin{equation*}
K_{1}f:=\int_{\mathbb{R}^{3}\times S^{2}}B(\vartheta )|\xi -\xi _{\ast
}|^{\gamma }\mathcal{M}^{1/2}(\xi )\mathcal{M}^{1/2}(\xi _{\ast })f(\xi
_{\ast })d\xi _{\ast }d\om \text{,}
\end{equation*}%
\begin{align*}
K_{2}f& :=\int_{\mathbb{R}^{3}\times S^{2}}B(\vartheta )|\xi -\xi _{\ast
}|^{\gamma }\mathcal{M}^{1/2}(\xi _{\ast })\mathcal{M}^{1/2}(\xi ^{\prime
})f(\xi _{\ast }^{\prime })d\xi _{\ast }d\om \\
& \quad +\int_{\mathbb{R}^{3}\times S^{2}}B(\vartheta )|\xi -\xi _{\ast
}|^{\gamma }\mathcal{M}^{1/2}(\xi _{\ast })\mathcal{M}^{1/2}(\xi _{\ast
}^{\prime })f(\xi ^{\prime })d\xi _{\ast }d\om \text{.}
\end{align*}

To begin with, we present a number of basic properties and estimates of the
operators $\mathcal{L}$, $\nu (\xi )$ and $K$, which can be found in \cite%
{[Grad], [LinWangWu]}.

\begin{lemma}
\label{basic} Let $0\leq \gamma \leq 1$. For any $g\in L_{\sigma }^{2}$, we
have the coercivity of the linearized collision operator $\mathcal{L}$, that
is, there exists a positive constant $\nu _{0}$ such that
\begin{equation*}
\left \langle g,\mathcal{L}g\right \rangle _{\xi }\leq -\nu _{0}\left \vert
\mathrm{P}_{1}g\right \vert _{L_{\sigma }^{2}}^{2}\text{.}
\label{coercivity}
\end{equation*}%
For the multiplicative operator $\nu (\xi )$, there are positive constants $%
\nu _{0}$ and $\nu _{1}$ such that
\begin{equation*}
\nu _{0}\left \langle \xi \right \rangle ^{\gamma }\leq \nu (\xi )\leq \nu
_{1}\left \langle \xi \right \rangle ^{\gamma }\text{.}  \label{nu-gamma}
\end{equation*}%
For the integral operator $K$,
\begin{equation*}
Kf=-K_{1}f+K_{2}f=\int_{{\mathbb{R}}^{3}}-k_{1}(\xi ,\xi _{\ast })f(\xi
_{\ast })d\xi _{\ast }+\int_{{\mathbb{R}}^{3}}k_{2}(\xi ,\xi _{\ast })f(\xi
_{\ast })d\xi _{\ast }\text{,}
\end{equation*}%
the kernels $k_{1}(\xi ,\xi _{\ast })$ and $k_{2}(\xi ,\xi _{\ast })$
satisfy
\begin{equation*}
k_{1}(\xi ,\xi _{\ast })\lesssim |\xi -\xi _{\ast }|^{\gamma }\exp \left \{ -%
\frac{1}{4}\left( |\xi |^{2}+|\xi _{\ast }|^{2}\right) \right \} \text{,}
\end{equation*}%
and%
\begin{equation*}
k_{2}(\xi ,\xi _{\ast })=a\left( \xi ,\xi _{\ast },\kappa \right) \exp
\left( -\frac{(1-\kappa )}{8}\left[ \frac{\left( \left \vert \xi \right
\vert ^{2}-\left \vert \xi _{\ast }\right \vert ^{2}\right) ^{2}}{\left
\vert \xi -\xi _{\ast }\right \vert ^{2}}+\left \vert \xi -\xi _{\ast
}\right \vert ^{2}\right] \right) \text{,}
\end{equation*}%
for any $0<\kappa <1$, together with
\begin{equation*}
a(\xi ,\xi _{\ast },\kappa )\leq C_{\kappa }|\xi -\xi _{\ast }|^{-1}(1+|\xi
|+|\xi _{\ast }|)^{\gamma -1}\text{.}
\end{equation*}%
\end{lemma}

Furthermore, from Lemma \ref{basic} we have some essential properties for
the integral operator $K$.

\begin{lemma}
\label{basic2}Let $0\leq \gamma \leq 1$ and $\tau \in \mathbb{R}$. Then
\begin{equation}
|Kg|_{L_{\xi ,\tau +2-\gamma }^{q}}\lesssim |g|_{L_{\xi ,\tau }^{q}}\text{, }%
1\leq q\leq \infty \text{,}  \label{K-Lp}
\end{equation}%
and
\begin{equation*}
|Kg|_{L_{\xi ,\tau -\gamma +3/2}^{\infty }}\leq C|g|_{L_{\xi ,\tau }^{2}}%
\text{\thinspace .}
\end{equation*}%
Moreover,
\begin{equation}
|\varpi K\varpi ^{-1}g|_{L_{\xi ,\tau +2-\gamma }^{q}}\lesssim |g|_{L_{\xi
,\tau }^{q}}\text{, }1\leq q\leq \infty \text{,}  \label{kaK-Lp}
\end{equation}%
holds for any weight function of the form
\begin{equation}
\varpi (\xi )=\exp \left \{ \ka_{0}|\xi |^{2}+\ka \cdot \xi +\ka_{4}\right \}
\label{weight}
\end{equation}%
with $\ka_{0}$, $\ka$, $\ka_{4}$ constant and $\ka_{0}>0$.
\end{lemma}

\begin{lemma}
\label{basic-Gamma}Let $0\leq \gamma \leq 1$ and $\tau \geq 0$. Then
\begin{equation*}
\left| \nu ^{-1}\varpi \Gamma (h_{1},h_{2})\right| _{L_{\xi ,\tau }^{\infty
}}\leq \left|\varpi h_{1}\right|_{L_{\xi ,\tau }^{\infty }}\left|\varpi
h_{2}\right|_{L_{\xi ,\tau }^{\infty }}\text{,}
\end{equation*}%
for any weight function defined as $\left( \ref{weight}\right) $.
\end{lemma}

\begin{theorem}[{The large time behavior for $0\leq \protect \gamma \leq 1$,
\protect \cite{[LinWangLyuWu]}}]
\label{prop: nonlinear}Let $\beta >3/2+\gamma $ and let $\varpi =\varpi (\xi
)$ be any weight function defined as $\left( \ref{weight}\right) $. Assume
that the initial data $\eps f_{0}^{\alpha }$ satisfies $\varpi f_{0}^{\alpha
}\in L_{\xi ,\beta }^{\infty }(L_{x}^{1}\cap L_{x}^{\infty })$ and $%
\varepsilon >0$ is sufficiently small. Then there is a unique solution $%
f^{\alpha }$ to $(\ref{nonlinear})$ in $L_{\xi ,\beta }^{\infty }(\varpi
)(L_{x}^{2}\cap L_{x}^{\infty })$ with
\begin{eqnarray*}
\left \Vert \varpi f^{\alpha }(t)\right \Vert _{L_{\xi ,\beta }^{\infty
}L_{x}^{2}} &\leq &\eps C_{1}(1+t)^{-\frac{3}{4}}\left \Vert \varpi
f_{0}^{\alpha }\right \Vert _{L_{\xi ,\beta }^{\infty }(L_{x}^{1}\cap
L_{x}^{\infty })}\text{,}%
\vspace {3mm}
\\
\left \Vert \varpi f^{\alpha }(t)\right \Vert _{L_{\xi ,\beta }^{\infty
}L_{x}^{\infty }} &\leq &\eps C_{2}(1+t)^{-\frac{3}{2}}\left \Vert \varpi
f_{0}^{\alpha }\right \Vert _{L_{\xi ,\beta }^{\infty }(L_{x}^{1}\cap
L_{x}^{\infty })}\text{.}
\end{eqnarray*}%
for some positive constants $C_{1}$ and $C_{2}$.
\end{theorem}

\subsection{Refined estimate for the linearized Boltzmann equation}

\label{refine} Let $u$ be the solution of the linearized Boltzmann equation
\begin{equation}
\left \{
\begin{array}{l}
\partial _{t}u+\xi \cdot \nabla _{x}u=\mathcal{L}u\text{,}%
\vspace {3mm}
\\
u\left( 0,x,\xi \right) =u_{0}\left( x,\xi \right) \text{,}%
\end{array}%
\right.  \label{Linearied Bolt}
\end{equation}%
and $\mathbb{G}^{t}$ the corresponding solution operator, i.e., $u=\mathbb{G}%
^{t}u_{0}$. Then we have known that

\begin{theorem}[{\protect \cite{[LinWangLyuWu], [Zhong]}}]
\label{Thm for linear}Let $u$ be the solution of $\left( \ref{Linearied Bolt}%
\right) $ and $\beta >3/2$. Then
\begin{equation*}
\left \Vert \mathbb{G}^{t}u_{0}\right \Vert _{L_{\xi ,\beta }^{\infty
}L_{x}^{\infty }}\leq C\left( 1+t\right) ^{-3/4}\left[ \left \Vert
u_{0}\right \Vert _{L_{\xi ,\beta }^{\infty }L_{x}^{\infty }}+\left \Vert
u_{0}\right \Vert _{L_{\xi ,\beta }^{\infty }L_{x}^{2}}\right] \text{,}
\end{equation*}%
\begin{equation*}
\left \Vert \mathbb{G}^{t}u_{0}\right \Vert _{L_{\xi ,\beta }^{\infty
}L_{x}^{2}}\leq C\left[ \left \Vert u_{0}\right \Vert _{L_{\xi ,\beta
}^{\infty }L_{x}^{2}}\right] \text{,}
\end{equation*}%
for $u_{0}\in L_{\xi ,\beta }^{\infty }\left( L_{x}^{\infty }\cap
L_{x}^{2}\right) $, and
\begin{equation*}
\left \Vert \mathbb{G}^{t}u_{0}\right \Vert _{L_{\xi ,\beta }^{\infty
}L_{x}^{\infty }}\leq C\left( 1+t\right) ^{-3/2}\left[ \left \Vert
u_{0}\right \Vert _{L_{\xi ,\beta }^{\infty }L_{x}^{\infty }}+\left \Vert
u_{0}\right \Vert _{L_{\xi ,\beta }^{\infty }L_{x}^{1}}\right] \text{,}
\end{equation*}%
\begin{equation*}
\left \Vert \mathbb{G}^{t}u_{0}\right \Vert _{L_{\xi ,\beta }^{\infty
}L_{x}^{2}}\leq C\left( 1+t\right) ^{-3/4}\left[ \left \Vert u_{0}\right
\Vert _{L_{\xi ,\beta }^{\infty }L_{x}^{2}}+\left \Vert u_{0}\right
\Vert _{L_{\xi ,\beta }^{\infty }L_{x}^{1}}\right] \text{,}
\end{equation*}%
for $u_{0}\in L_{\xi ,\beta }^{\infty }\left( L_{x}^{\infty }\cap
L_{x}^{1}\right) $. Moreover, if $\mathrm{P}_{0}u_{0}=0$, then we will get
extra $(1+t)^{-1/2}$ decay rate in each estimate above.
\end{theorem}

To attain the desired time decay rate, we need to refine these estimates for
the linearized Boltzmann equation. According to the semigroup theory, the
solution $u$ to$\  \left( \ref{Linearied Bolt}\right) $ can be represented by
\begin{equation*}
u\left( t,x,\xi \right) =\mathbb{G}^{t}u_{0}=\left( 2\pi \right) ^{-3}\int_{%
\mathbb{R}^{3}}e^{ix\cdot \eta }e^{\left( -i\xi \cdot \eta +\mathcal{L}%
\right) t}\widehat{u}_{0}\left( \eta \right) d\eta \text{,}
\end{equation*}%
where $\widehat{u}_{0}$ is the Fourier transformation of $u_{0}$ with
respect to the space variable $x$. Based on the spectrum analysis of the
operator $-i\xi \cdot \eta +\mathcal{L}$, the semigroup $e^{\left( -i\xi
\cdot \eta +\mathcal{L}\right) t}$ can be decomposed as
\begin{eqnarray*}
e^{\left( -i\xi \cdot \eta +\mathcal{L}\right) t} &=&\chi _{\delta }\left(
\eta \right) \sum_{j=0}^{4}e^{\lambda _{j}\left( \eta \right) t}\left \vert
e_{j}\left( \eta \right) \big>\big<e_{j}\left( \eta \right) \right \vert
+\chi _{\delta }\left( \eta \right) e^{\left( -i\xi \cdot \eta +\mathcal{L}%
\right) t}\Pi _{\eta }^{\bot }+\left( 1-\chi _{\delta }\left( \eta \right)
\right) e^{\left( -i\xi \cdot \eta +\mathcal{L}\right) t}%
\vspace {3mm}
\\
&=&:\widehat{\mathbb{G}}_{L;0}\left( \eta ,t\right) +\widehat{\mathbb{G}}%
_{L;\bot }\left( \eta ,t\right) +\widehat{\mathbb{G}}_{S}\left( \eta
,t\right)
\end{eqnarray*}%
where $\chi _{\delta }\left( \eta \right) $ is a smooth cutoff function with
$0\leq \chi _{\delta }\leq 1$, $\chi _{\delta }\left( \eta \right) =1$ for $%
\left \vert \eta \right \vert \leq \frac{\delta }{2}$ and $0$ for $%
\left
\vert \eta \right \vert \geq \delta $, for $\delta >0$ sufficient
small. Note that the spectrums of $\widehat{\mathbb{G}}_{L;\bot }\left( \eta
,t\right) $ and $\widehat{\mathbb{G}}_{S}\left( \eta ,t\right) $ are
strictly away from imaginary axis (with negative real part). Moreover for $\left \vert
\eta \right \vert \ll 1$, the spectrum Spec $\left( \eta \right) $ of the
operator $-i\xi \cdot \eta +\mathcal{L}$ consists of exactly five
eigenvalues $\lambda _{j}\left( \eta \right) $ ($0\leq j\leq 4$) associated
with the corresponding eigenfunctions $e_{j}\left( \eta \right) $ (\cite{[EP],
[LiuYu1], [YangYu]}):
\begin{equation*}
\begin{array}{l}
\displaystyle \lambda _{j}(\eta )=-i\,a_{j}|\eta |-A_{j}|\eta |^{2}+O(|\eta
|^{3})\text{,} \\
\\
\displaystyle e_{j}(\eta )=E_{j}+O(|\eta |)\text{,}%
\end{array}%
\end{equation*}%
here $A_{j}>0$, $\left \langle e_{j}(-\eta ),e_{l}(\eta )\right \rangle
_{\xi }=\delta _{jl}$, $0\leq j$, $l\leq 4$ and
\begin{equation*}
\left \{
\begin{array}{l}
a_{0}=\sqrt{\frac{5}{3}}\text{,}\quad a_{1}=-\sqrt{\frac{5}{3}}\text{,}\,
\quad a_{2}=a_{3}=a_{4}=0\text{,} \\[2mm]
E_{0}=\sqrt{\frac{3}{10}}\chi _{0}+\sqrt{\frac{1}{2}}\om \cdot \overline{%
\chi }+\sqrt{\frac{1}{5}}\chi _{4}\text{,} \\[2mm]
E_{1}=\sqrt{\frac{3}{10}}\chi _{0}-\sqrt{\frac{1}{2}}\om \cdot \overline{%
\chi }+\sqrt{\frac{1}{5}}\chi _{4}\text{,} \\[2mm]
E_{2}=-\sqrt{\frac{2}{5}}\chi _{0}+\sqrt{\frac{3}{5}}\chi _{4}\, \text{,} \\%
[2mm]
E_{3}=\om_{1}\cdot \overline{\chi }\text{,} \\[2mm]
E_{4}=\om_{2}\cdot \overline{\chi }\text{,}%
\end{array}%
\right.
\end{equation*}%
where $\overline{\chi }=(\chi _{1},\chi _{2},\chi _{3})$, $\eta =\left \vert
\eta \right \vert \om$ ($\om \in S^{2}$) and $\{ \om_{1},\om_{2},\om \}$ is
an orthonormal basis of ${\mathbb{R}}^{3}$. And then we define
\begin{equation*}
u_{L;0}=\mathbb{G}_{L;0}^{t}u_{0}=\left( 2\pi \right) ^{-3}\int_{\mathbb{R}%
^{3}}e^{ix\cdot \eta }\widehat{\mathbb{G}}_{L;0}\left( \eta ,t\right)
\widehat{u}_{0}\left( \eta \right) d\eta \text{,}
\end{equation*}%
\begin{equation*}
u_{L;\bot }=\mathbb{G}_{L;\bot }^{t}u_{0}=\left( 2\pi \right) ^{-3}\int_{%
\mathbb{R}^{3}}e^{ix\cdot \eta }\widehat{\mathbb{G}}_{L;\bot }\left( \eta
,t\right) \widehat{u}_{0}\left( \eta \right) d\eta \text{,}
\end{equation*}%
\begin{equation*}
u_{S}=\mathbb{G}_{S}^{t}u_{0}=\left( 2\pi \right) ^{-3}\int_{\mathbb{R}%
^{3}}e^{ix\cdot \eta }\widehat{\mathbb{G}}_{S}\left( \eta ,t\right) \widehat{%
u}_{0}\left( \eta \right) d\eta \text{,}
\end{equation*}%
called the fluid part and nonfluid part of the long wave of $u$, and the
short wave of $u$, respectively. In the meanwhile, we define%
\begin{equation*}
u_{0L}:=\left( 2\pi \right) ^{-3}\int_{\mathbb{R}^{3}}e^{ix\cdot \eta }\chi
_{\delta }\left( \eta \right) \widehat{u}_{0}\left( \eta \right) d\eta \text{%
,}
\end{equation*}%
\begin{equation*}
u_{0S}:=\left( 2\pi \right) ^{-3}\int_{\mathbb{R}^{3}}e^{ix\cdot \eta
}\left( 1-\chi _{\delta }\left( \eta \right) \right) \widehat{u}_{0}\left(
\eta \right) d\eta \text{.}
\end{equation*}%
One can see that the long wave $u_{L}:=u_{L;0}+u_{L;\bot }$ satisfies the
equation%
\begin{equation}
\left \{
\begin{array}{l}
\partial _{t}u_{L}+\xi \cdot \nabla _{x}u_{L}=\mathcal{L}u_{L}\text{,}%
\vspace {3mm}
\\
u_{L}\left( 0,x,\xi \right) =u_{0L}\left( x,\xi \right) \text{,}%
\end{array}%
\right.  \label{Eqn-long}
\end{equation}%
and the short wave satisfies%
\begin{equation}
\left \{
\begin{array}{l}
\partial _{t}u_{S}+\xi \cdot \nabla _{x}u_{S}=\mathcal{L}u_{S}\text{,}%
\vspace {3mm}
\\
u_{S}\left( 0,x,\xi \right) =u_{0S}\left( x,\xi \right) \text{.}%
\end{array}%
\right.  \label{Eqn-short}
\end{equation}

According to \cite{[LinWangWu], [LiuYu2]}, the wave structure is given by
\begin{equation*}
\left \Vert \partial _{x}^{\alpha }\mathbb{G}_{L;0}^{t}\left( x,t\right)
\right \Vert \leq C_{N}(1+t)^{-|\alpha |/2}\left[
\begin{array}{l}
\left( 1+t\right) ^{-2}B_{N}\left( \left \vert x\right \vert -\mathbf{c}%
t,t\right) +\left( 1+t\right) ^{-3/2}B_{N}\left( \left \vert x\right \vert
,t\right) \\[2mm]
+\mathbf{1}_{\{ \left \vert x\right \vert \leq \mathbf{c}t\}}\left(
1+t\right) ^{-3/2}B_{3/2}\left( \left \vert x\right \vert ,t\right)%
\end{array}%
\right]
\end{equation*}%
for all $N\in \mathbb{N}$ if $0\leq \gamma \leq 1$, where $\alpha =\left(
\alpha _{1},\alpha _{2},\alpha _{3}\right) $ is a multi-index with $%
\left
\vert \alpha \right \vert \geq 0$, $\left \Vert \cdot \right \Vert $
denotes the operator norm from $L_{\xi }^{2}$ to $L_{\xi }^{2}$, the number $%
\mathbf{c=}\sqrt{\frac{5}{3}}$ is the sound speed, $\mathbf{1}_{D}$ is the
characteristic function of the set $D$ and
\begin{equation*}
B_{N}\left( z,t\right) =\left( 1+\frac{z^{2}}{1+t}\right) ^{-N}\text{.}\,
\end{equation*}%
We then have the following proposition:

\begin{proposition}
\label{Prop-LS}Let $u$ be a solution to (\ref{Linearied Bolt}) with the
initial data $u_{0}$. Then
\begin{equation}
\left \Vert \partial _{x}^{\alpha }\mathbb{G}_{L;0}^{t}u_{0}\right \Vert
_{L_{x}^{q}\left( L_{\xi }^{2}\right) }\leq C\left( 1+t\right) ^{-\frac{3}{2}%
\left( \frac{1}{p}-\frac{1}{q}\right) -\frac{\left \vert \alpha \right \vert
}{2}}\left \Vert u_{0}\right \Vert _{L_{x}^{p}L_{\xi }^{2}}\text{,}\quad
1\leq p\leq q\leq \infty \, \text{,}  \label{long-fluid}
\end{equation}%
\begin{equation}
\left \Vert \partial _{x}^{\alpha }\mathbb{G}_{L;\bot }^{t}u_{0}\right \Vert
_{L_{x}^{2}L_{\xi }^{2}}\leq Ce^{-\frac{t}{c}}\left \Vert u_{0}\right \Vert
_{L_{x}^{2}L_{\xi }^{2}}\, \text{,}  \label{long-nonfluid}
\end{equation}%
\begin{equation}
\left \Vert \mathbb{G}_{S}^{t}u_{0}\right \Vert _{L_{x}^{2}L_{\xi }^{2}}\leq
Ce^{-\frac{t}{c}}\left \Vert u_{0}\right \Vert _{L_{x}^{2}L_{\xi }^{2}}\,
\text{,}  \label{short}
\end{equation}%
and moreover, if $\mathrm{P}_{0}u_{0}=0$, then
\begin{equation*}
\left \Vert \partial _{x}^{\alpha }\mathbb{G}_{L;0}^{t}u_{0}\right \Vert
_{L_{x}^{q}\left( L_{\xi }^{2}\right) }\leq C\left( 1+t\right) ^{-\frac{3}{2}%
\left( \frac{1}{p}-\frac{1}{q}\right) -\frac{1}{2}-\frac{\left \vert \alpha
\right \vert }{2}}\left \Vert u_{0}\right \Vert _{L_{x}^{p}L_{\xi }^{2}}%
\text{\thinspace ,}\quad 1\leq p\leq q\leq \infty \text{\thinspace ,}
\end{equation*}%
for some constants $C$, $c>0$, where $\alpha =\left( \alpha _{1},\alpha
_{2},\alpha _{3}\right) $ is a multi-index with $\left \vert \alpha
\right
\vert \geq 0$.
\end{proposition}

In view of $\left( \ref{long-nonfluid}\right) $, we have%
\begin{equation*}
\left \Vert \mathbb{G}_{L;\bot }^{t}u_{0}\right \Vert _{L_{x}^{r}L_{\xi
}^{2}}\leq Ce^{-\frac{t}{c}}\left \Vert u_{0}\right \Vert _{L_{x}^{2}L_{\xi
}^{2}}
\end{equation*}%
for $2\leq r\leq \infty $. However, we have the $L_{x}^{2}$ estimate for $%
\mathbb{G}_{S}^{t}u_{0}$ only. In order to obtain the $L_{x}^{\infty }$
estimate, we apply the singular-regular decomposition as those in \cite%
{[LinWangWu], [LiuYu]}: We denote the solution operator of the damped
transport equation
\begin{equation*}
\left \{
\begin{array}{l}
\partial _{t}h+\xi \cdot \nabla _{x}h+\nu (\xi )h=0\text{,} \\[2mm]
h(0,x,\xi )=h_{0}\text{,}%
\end{array}%
\right.
\end{equation*}%
by $\mathbb{S}^{t}$, i.e., $h(t)=\mathbb{S}^{t}h_{0}$. Then we design the
series as
\begin{equation*}
\mathbb{G}^{t}u_{0}=\sum_{j=0}^{m}u^{(j)}+R^{\left( m\right)
}=W^{(m)}+R^{\left( m\right) }\text{,}
\end{equation*}%
for some $m\in \mathbb{N}$ (precisely, $m\geq 6$), where $u^{\left( j\right)
}$ and $R^{\left( m\right) }$ are defined by%
\begin{equation*}
u^{(0)}=\mathbb{S}^{t}u_{0}\, \text{,}\quad u^{(j)}=\int_{0}^{t}\mathbb{S}%
^{t-\tau }Ku^{(j-1)}(\tau )d\tau \text{,}\quad 1\leq j\leq m\text{,}
\end{equation*}%
and
\begin{equation*}
R^{\left( m\right) }=\int_{0}^{t}\mathbb{G}^{t-\tau }Ku^{(m)}(\tau )d\tau
\text{.}
\end{equation*}

Combining the singular-regular decomposition with Proposition \ref{Prop-LS},
we are able to get the $L_{x}^{r}$ estimate of $\mathbb{G}^{t}u_{0}$ for $%
2\leq r\leq \infty $:

\begin{proposition}
\label{Prop-linear}Let $u$ be a solution of $\left( \ref{Linearied Bolt}%
\right) $ with the initial data $u_{0}$. Then
\begin{equation*}
\left \Vert \mathbb{G}^{t}u_{0}\right \Vert _{L_{x}^{r}L_{\xi }^{2}}\lesssim
\left( 1+t\right) ^{-\frac{3}{2}\left( \frac{1}{p}-\frac{1}{r}\right) }\left
\Vert u_{0}\right \Vert _{L_{x}^{p}L_{\xi }^{2}}+e^{-\frac{t}{c}}\left(
\left \Vert u_{0}\right \Vert _{L_{\xi }^{2}L_{x}^{2}}+\left \Vert
u_{0}\right \Vert _{L_{\xi }^{2}L_{x}^{2}}^{2/r}\left \Vert u_{0}\right
\Vert _{L_{\xi }^{2}L_{x}^{\infty }}^{1-2/r}\right) \text{,}
\end{equation*}%
\begin{equation*}
\left \Vert \mathbb{G}^{t}u_{0}\right \Vert _{L_{\xi ,\beta }^{\infty
}L_{x}^{r}}\lesssim \left( 1+t\right) ^{-\frac{3}{2}\left( \frac{1}{p}-\frac{%
1}{r}\right) }\left \Vert u_{0}\right \Vert _{L_{x}^{p}L_{\xi }^{2}}+e^{-%
\frac{t}{c}}\left( \left \Vert u_{0}\right \Vert _{L_{\xi ,\beta }^{\infty
}L_{x}^{r}}+\left \Vert u_{0}\right \Vert _{L_{\xi }^{2}L_{x}^{2}}+\left
\Vert u_{0}\right \Vert _{L_{\xi }^{2}L_{x}^{2}}^{2/r}\left \Vert
u_{0}\right \Vert _{L_{\xi }^{2}L_{x}^{\infty }}^{1-2/r}\right) \text{,}
\end{equation*}%
\begin{equation*}
\left \Vert \varpi \mathbb{G}^{t}u_{0}\right \Vert _{L_{\xi ,\beta }^{\infty
}L_{x}^{r}}\lesssim \left( 1+t\right) ^{-\frac{3}{2}\left( \frac{1}{p}-\frac{%
1}{r}\right) }\left \Vert u_{0}\right \Vert _{L_{x}^{p}L_{\xi }^{2}}+e^{-%
\frac{t}{c}}\left( \left \Vert \varpi u_{0}\right \Vert _{L_{\xi ,\beta
}^{\infty }L_{x}^{r}}+\left \Vert u_{0}\right \Vert _{L_{\xi
}^{2}L_{x}^{2}}+\left \Vert u_{0}\right \Vert _{L_{\xi
}^{2}L_{x}^{2}}^{2/r}\left \Vert u_{0}\right \Vert _{L_{\xi
}^{2}L_{x}^{\infty }}^{1-2/r}\right) \text{,}
\end{equation*}%
for $\beta \geq 0$, $2\leq r\leq \infty $ and $1\leq p\leq r$. Moreover, if $%
\mathrm{P}_{0}u_{0}=0$, then we will get extra $(1+t)^{-1/2}$ decay rate in
the above estimates.
\end{proposition}

\begin{proof}
By Lemma \ref{basic2},
\begin{equation}
\left \Vert W^{(m)}\right \Vert _{L_{\xi }^{q}L_{x}^{p}}\lesssim e^{-\frac{t%
}{c}}\left \Vert u_{0}\right \Vert _{L_{\xi }^{q}L_{x}^{p}}\text{, }1\leq p%
\text{\thinspace , }q\leq \infty \text{,}  \label{Wave-pq}
\end{equation}%
\begin{equation}
\left \Vert W^{(m)}\right \Vert _{L_{x}^{\infty }L_{\xi }^{q}}\lesssim e^{-%
\frac{t}{c}}\left \Vert u_{0}\right \Vert _{L_{\xi }^{q}L_{x}^{\infty }}%
\text{, }1\leq q\leq \infty \text{.}  \label{Wave-sup-q}
\end{equation}%
Utilizing the Mixture Lemma \cite{[LinWangWu], [LiuYu]} yields
\begin{equation}
\left \Vert R^{\left( m\right) }\right \Vert _{H_{x}^{2}L_{\xi }^{2}}=\left
\Vert R^{\left( m\right) }\right \Vert _{L_{\xi }^{2}H_{x}^{2}}\lesssim
\left \Vert u_{0}\right \Vert _{L_{\xi }^{2}L_{x}^{2}}\text{.}
\label{remainder}
\end{equation}%
Note that $\mathbb{G}^{t}u_{0}=\mathbb{G}_{L;0}^{t}u_{0}+\mathbb{G}_{L;\bot
}^{t}u_{0}+\mathbb{G}_{S}^{t}u_{0}=W^{(m)}+R^{\left( m\right) }$. In light
of Proposition \ref{Prop-LS}, $\left( \ref{Wave-pq}\right) $ and $\left( \ref%
{remainder}\right) $, we find
\begin{eqnarray*}
&&\left \Vert \mathbb{G}_{S}^{t}u_{0}-W^{\left( m\right) }\right \Vert
_{L_{x}^{\infty }L_{\xi }^{2}}\,,\left \Vert \mathbb{G}_{S}^{t}u_{0}-W^{%
\left( m\right) }\right \Vert _{L_{\xi }^{2}L_{x}^{\infty }} \\
&\lesssim &\left \Vert \mathbb{G}_{S}^{t}u_{0}-W^{\left( m\right) }\right
\Vert _{L_{\xi }^{2}L_{x}^{2}}^{3/4}\left \Vert \mathbb{G}%
_{S}^{t}u_{0}-W^{\left( m\right) }\right \Vert _{L_{\xi }^{2}H_{x}^{2}}^{1/4}
\\
&\lesssim &\left \Vert \mathbb{G}_{S}^{t}u_{0}-W^{\left( m\right) }\right
\Vert _{L_{\xi }^{2}L_{x}^{2}}^{3/4}\left \Vert \left( \mathbb{G}%
_{L;0}^{t}u_{0}+\mathbb{G}_{L;\bot }^{t}u_{0}\right) -R^{\left( m\right)
}\right \Vert _{L_{\xi }^{2}H_{x}^{2}}^{1/4}\lesssim e^{-\frac{t}{c}}\left
\Vert u_{0}\right \Vert _{L_{\xi }^{2}L_{x}^{2}}\text{,}
\end{eqnarray*}%
and so
\begin{equation*}
\left \Vert \mathbb{G}_{S}^{t}u_{0}\right \Vert _{L_{x}^{\infty }L_{\xi
}^{2}}\leq \left \Vert \mathbb{G}_{S}^{t}u_{0}-W^{\left( m\right) }\right
\Vert _{L_{x}^{\infty }L_{\xi }^{2}}+\left \Vert W^{\left( m\right) }\right
\Vert _{L_{x}^{\infty }L_{\xi }^{2}}\lesssim e^{-\frac{t}{c}}\left( \left
\Vert u_{0}\right \Vert _{L_{\xi }^{2}L_{x}^{2}}+\left \Vert u_{0}\right
\Vert _{L_{\xi }^{2}L_{x}^{\infty }}\right)
\end{equation*}%
due to $\left( \ref{Wave-sup-q}\right) $. By the interpolation with $\left( %
\ref{short}\right) $,%
\begin{equation}
\left \Vert \mathbb{G}_{S}^{t}u_{0}\right \Vert _{L_{x}^{r}L_{\xi
}^{2}}\lesssim e^{-\frac{t}{c}}\left( \left \Vert u_{0}\right \Vert _{L_{\xi
}^{2}L_{x}^{2}}+\left \Vert u_{0}\right \Vert _{L_{\xi
}^{2}L_{x}^{2}}^{2/r}\left \Vert u_{0}\right \Vert _{L_{\xi
}^{2}L_{x}^{\infty }}^{1-2/r}\right)  \label{short-r-out}
\end{equation}%
for $2\leq r\leq \infty $. Together with $\left( \ref{long-fluid}\right) $-$%
\left( \ref{long-nonfluid}\right) $, we conclude
\begin{equation}
\left \Vert \mathbb{G}^{t}u_{0}\right \Vert _{L_{x}^{r}L_{\xi }^{2}}\lesssim
\left( 1+t\right) ^{-\frac{3}{2}\left( \frac{1}{p}-\frac{1}{r}\right) }\left
\Vert u_{0}\right \Vert _{L_{x}^{p}L_{\xi }^{2}}+e^{-\frac{t}{c}}\left(
\left \Vert u_{0}\right \Vert _{L_{\xi }^{2}L_{x}^{2}}+\left \Vert
u_{0}\right \Vert _{L_{\xi }^{2}L_{x}^{2}}^{2/r}\left \Vert u_{0}\right
\Vert _{L_{\xi }^{2}L_{x}^{\infty }}^{1-2/r}\right)  \label{Linear-r-out}
\end{equation}%
for $2\leq r\leq \infty $ and $1\leq p\leq r$.

Next we will derive the weighted estimate. In view of $\left( \ref{Linearied
Bolt}\right) $,
\begin{equation*}
u=\mathbb{G}^{t}u_{0}=\mathbb{S}^{t}u_{0}+\int_{0}^{t}\mathbb{S}^{t-\tau
}Ku\left( \tau \right) d\tau \text{,}
\end{equation*}%
and then
\begin{eqnarray*}
\left \vert u\right \vert _{L_{x}^{r}} &\leq &\left \vert \mathbb{S}%
^{t}u_{0}\right \vert _{L_{x}^{r}}+\int_{0}^{t}\left \vert \mathbb{S}%
^{t-\tau }Ku\left( \tau \right) \right \vert _{L_{x}^{r}}d\tau \\
&\leq &e^{-\frac{t}{c}}\left \Vert u_{0}\right \Vert _{L_{\xi }^{\infty
}L_{x}^{r}}+\int_{0}^{t}e^{-\frac{t-\tau }{c}}\left \Vert Ku\left( \tau
\right) \right \Vert _{L_{\xi }^{\infty }L_{x}^{r}}d\tau \\
&\lesssim &e^{-\frac{t}{c}}\left \Vert u_{0}\right \Vert _{L_{\xi }^{\infty
}L_{x}^{r}}+\int_{0}^{t}e^{-\frac{t-\tau }{c}}\left \Vert u\left( \tau
\right) \right \Vert _{L_{x}^{r}L_{\xi }^{2}}d\tau \\
&\lesssim &e^{-\frac{t}{c}}\left \Vert u_{0}\right \Vert _{L_{\xi }^{\infty
}L_{x}^{r}}+\left( 1+t\right) ^{-\frac{3}{2}\left( \frac{1}{p}-\frac{1}{r}%
\right) }\left \Vert u_{0}\right \Vert _{L_{x}^{p}L_{\xi }^{2}}+e^{-\frac{t}{%
c}}\left( \left \Vert u_{0}\right \Vert _{L_{\xi }^{2}L_{x}^{2}}+\left \Vert
u_{0}\right \Vert _{L_{\xi }^{2}L_{x}^{2}}^{2/r}\left \Vert u_{0}\right
\Vert _{L_{\xi }^{2}L_{x}^{\infty }}^{1-2/r}\right)
\end{eqnarray*}%
for $2\leq r\leq \infty $ and $1\leq p\leq r$, by using $\left( \ref%
{Linear-r-out}\right) $. Through the bootstrap argument,
\begin{equation}
\left \Vert u\right \Vert _{L_{\xi ,\beta }^{\infty }L_{x}^{r}}\lesssim
\left( 1+t\right) ^{-\frac{3}{2}\left( \frac{1}{p}-\frac{1}{r}\right) }\left
\Vert u_{0}\right \Vert _{L_{x}^{p}L_{\xi }^{2}}+e^{-\frac{t}{c}}\left(
\left \Vert u_{0}\right \Vert _{L_{\xi ,\beta }^{\infty }L_{x}^{r}}+\left
\Vert u_{0}\right \Vert _{L_{\xi }^{2}L_{x}^{2}}+\left \Vert u_{0}\right
\Vert _{L_{\xi }^{2}L_{x}^{2}}^{2/r}\left \Vert u_{0}\right \Vert _{L_{\xi
}^{2}L_{x}^{\infty }}^{1-2/r}\right) \text{,}  \label{u-r-in}
\end{equation}%
for $\beta \geq 0$.

Furthermore,
\begin{eqnarray*}
\left \vert \varpi \left( \xi \right) \left \langle \xi \right \rangle ^{\beta
}u\right \vert _{L_{x}^{r}} &=&\left \vert \varpi \left( \xi \right)
\left \langle \xi \right \rangle ^{\beta }\mathbb{S}^{t}u_{0}+\int_{0}^{t}%
\varpi \left( \xi \right) \left \langle \xi \right \rangle ^{\beta }\mathbb{S}%
^{t-\tau }Ku\left( \tau \right) d\tau \right \vert _{L_{x}^{r}} \\
&\leq &\varpi \left( \xi \right) \left \langle \xi \right \rangle ^{\beta
}\left \vert \mathbb{S}^{t}u_{0}\right \vert _{L_{x}^{r}}+\int_{0}^{t}\varpi
\left( \xi \right) \left \langle \xi \right \rangle ^{\beta }\left \vert
\mathbb{S}^{t-\tau }Ku\left( \tau \right) \right \vert _{L_{x}^{r}}d\tau  \\
&\equiv &\left( I\right) +\left( II\right) \text{.}
\end{eqnarray*}%
It readily follows that
\begin{equation*}
\left( I\right) \lesssim e^{-\frac{t}{c}}\left \Vert \varpi u_{0}\right \Vert
_{L_{\xi ,\beta }^{\infty }L_{x}^{r}}\text{.}
\end{equation*}%
For $\left( II\right) $, we split the integrand into two parts $\left \vert
\xi \right \vert \leq \theta $ and $\left \vert \xi \right \vert >\theta $,
where $\theta >0$ will be determined later. That is, for any $\beta \geq 0$,%
\begin{eqnarray*}
&&\varpi \left( \xi \right) \left \langle \xi \right \rangle ^{\beta
}\left \vert \mathbb{S}^{t-\tau }Ku\left( \tau \right) \right \vert
_{L_{x}^{r}} \\
&\leq &e^{-\nu \left( \xi \right) \left( t-\tau \right) }\left[
\sup_{\left \vert \xi \right \vert \leq \theta }\varpi \left( \xi \right)
\left \langle \xi \right \rangle ^{\beta }\left \vert Ku\left( \tau \right)
\right \vert _{L_{x}^{r}}+\sup_{\left \vert \xi \right \vert >\theta }\varpi
\left( \xi \right) \left \langle \xi \right \rangle ^{\beta }\left \vert
Ku\left( \tau \right) \right \vert _{L_{x}^{r}}\right]  \\
&\lesssim &e^{-\frac{t-\tau }{c}}\left[ e^{c_{1}\left \vert \theta
\right \vert ^{2}}\left \Vert u\left( \tau \right) \right \Vert _{L_{\xi ,\beta
}^{\infty }L_{x}^{r}}+\left( 1+\theta \right) ^{\gamma -2}\left \Vert \varpi
Ku\left( \tau \right) \right \Vert _{L_{\xi ,2-\gamma +\beta }^{\infty
}L_{x}^{r}}\right]  \\
&\lesssim &e^{-\frac{t-\tau }{c}}\left[ e^{c_{1}\left \vert \theta
\right \vert ^{2}}\left \Vert u\left( \tau \right) \right \Vert _{L_{\xi ,\beta
}^{\infty }L_{x}^{r}}+\left( 1+\theta \right) ^{\gamma -2}\left \Vert \varpi
K\varpi ^{-1}\varpi u\left( \tau \right) \right \Vert _{L_{\xi ,2-\gamma
+\beta }^{\infty }L_{x}^{r}}\right]  \\
&\lesssim &e^{-\frac{t-\tau }{c}}\left[ e^{c_{1}\left \vert \theta
\right \vert ^{2}}\left \Vert u\left( \tau \right) \right \Vert _{L_{\xi ,\beta
}^{\infty }L_{x}^{r}}+\left( 1+\theta \right) ^{\gamma -2}\left \Vert \varpi
u\left( \tau \right) \right \Vert _{L_{\xi ,\beta }^{\infty }L_{x}^{r}}\right]
\end{eqnarray*}%
by $\left( \ref{kaK-Lp}\right) $ and so%
\begin{equation*}
\left( II\right) \lesssim \int_{0}^{t}e^{-\frac{t-\tau }{c}}\left[
e^{c_{2}\left \vert \theta \right \vert ^{2}}\left \Vert u\left( \tau \right)
\right \Vert _{L_{\xi ,\beta }^{\infty }L_{x}^{r}}+\left( 1+\theta \right)
^{\gamma -2}\left \Vert \varpi u\left( \tau \right) \right \Vert _{L_{\xi
,\beta }^{\infty }L_{x}^{r}}\right] d\tau \text{.}
\end{equation*}%
From $\left( \ref{u-r-in}\right) $ it follows that
\begin{align*}
& \quad \left \Vert \varpi u\left( t\right) \right \Vert _{L_{\xi ,\beta
}^{\infty }L_{x}^{r}} \\
& \lesssim e^{-\frac{t}{c}}\left[ \left \Vert \varpi u_{0}\right \Vert
_{L_{\xi ,\beta }^{\infty }L_{x}^{r}}+C\left( \theta \right) \left(
\left \Vert u_{0}\right \Vert _{L_{\xi ,\beta }^{\infty }L_{x}^{r}}+\left \Vert
u_{0}\right \Vert _{L_{\xi }^{2}L_{x}^{2}}+\left \Vert u_{0}\right \Vert
_{L_{\xi }^{2}L_{x}^{2}}^{2/r}\left \Vert u_{0}\right \Vert _{L_{\xi
}^{2}L_{x}^{\infty }}^{1-2/r}\right) \right]  \\
& \quad +C\left( \theta \right) \left( 1+t\right) ^{-\frac{3}{2}\left( \frac{%
1}{p}-\frac{1}{r}\right) }\left \Vert u_{0}\right \Vert _{L_{x}^{p}L_{\xi
}^{2}}+\left( 1+\theta \right) ^{\gamma -2}\int_{0}^{t}e^{-\frac{t-\tau }{c}%
}\left \Vert \varpi u\left( \tau \right) \right \Vert _{L_{\xi ,\beta
}^{\infty }L_{x}^{r}}d\tau \text{.}
\end{align*}%
After $\theta $ is chosen sufficiently large,
\begin{align*}
& \lVert \varpi u\left( t\right) \rVert _{L_{\xi ,\beta }^{\infty }L_{x}^{r}}
\\
& \lesssim \left( 1+t\right) ^{-\frac{3}{2}\left( \frac{1}{p}-\frac{1}{r}%
\right) }\left \Vert u_{0}\right \Vert _{L_{x}^{p}L_{\xi }^{2}}+e^{-t/c}\left(
\left \Vert \varpi u_{0}\right \Vert _{L_{\xi ,\beta }^{\infty
}L_{x}^{r}}+\left \Vert u_{0}\right \Vert _{L_{\xi }^{2}L_{x}^{2}}+\left \Vert
u_{0}\right \Vert _{L_{\xi }^{2}L_{x}^{2}}^{2/r}\left \Vert u_{0}\right \Vert
_{L_{\xi }^{2}L_{x}^{\infty }}^{1-2/r}\right)
\end{align*}%
holds for $\beta \geq 0$, $2\leq r\leq \infty $ and $1\leq p\leq r$.
\end{proof}

In view of $\left( \ref{Eqn-long}\right) $, we see $u_{L}=\mathbb{G}%
^{t}u_{0L}$. Observe that%
\begin{eqnarray*}
\left \Vert u_{0L}\right \Vert _{L_{\xi ,\beta }^{\infty }L_{x}^{\infty }}
&=&\left \Vert \left( 2\pi \right) ^{-3}\int_{\mathbb{R}^{3}}e^{ix\cdot \eta
}\chi _{\delta }\left( \eta \right) \widehat{u}_{0}\left( \eta \right) d\eta
\right \Vert _{L_{\xi ,\beta }^{\infty }L_{x}^{\infty }} \\
&\leq &\left( 2\pi \right) ^{-3}\left \vert \int_{\left \vert \eta \right
\vert <\delta }\left \vert \chi _{\delta }\left( \eta \right) \widehat{u}%
_{0}\left( \eta \right) \right \vert d\eta \right \vert _{L_{\xi ,\beta
}^{\infty }}\lesssim \left \Vert u_{0L}\right \Vert _{L_{\xi ,\beta
}^{\infty }L_{x}^{2}}
\end{eqnarray*}%
and similarly for $\left \Vert \varpi u_{0L}\right \Vert _{L_{\xi ,\beta
}^{\infty }L_{x}^{\infty }}$. Hence $\left \Vert \varpi u_{0L}\right \Vert
_{L_{\xi ,\beta }^{\infty }L_{x}^{r}}\lesssim \left \Vert \varpi
u_{0L}\right \Vert _{L_{\xi ,\beta }^{\infty }L_{x}^{2}}$ and $\left \Vert
u_{0L}\right \Vert _{L_{\xi ,\beta }^{\infty }L_{x}^{r}}\lesssim \left \Vert
u_{0L}\right \Vert _{L_{\xi ,\beta }^{\infty }L_{x}^{2}}$ for $2\leq r\leq
\infty $. It immediately follows from Proposition \ref{Prop-linear} that

\begin{corollary}
\label{lemm-long}Let $u$ be a solution of $\left( \ref{Linearied Bolt}%
\right) $ with initial data $u_{0}$. Then $u_{L}=\mathbb{G}^{t}u_{0L}$
satisfies
\begin{equation*}
\left \Vert u_{L}\right \Vert _{L_{x}^{r}L_{\xi }^{2}}\lesssim \left(
1+t\right) ^{-\frac{3}{2}\left( \frac{1}{p}-\frac{1}{r}\right) }\left \Vert
u_{0L}\right \Vert _{L_{x}^{p}L_{\xi }^{2}}+e^{-\frac{t}{c}}\left( \left
\Vert u_{0L}\right \Vert _{L_{x}^{2}L_{\xi }^{2}}+\left \Vert u_{0L}\right
\Vert _{L_{x}^{2}L_{\xi }^{2}}^{2/r}\left \Vert u_{0L}\right \Vert _{L_{\xi
}^{2}L_{x}^{\infty }}^{1-2/r}\right) \text{,}
\end{equation*}%
\begin{equation*}
\left \Vert u_{L}\right \Vert _{L_{\xi ,\beta }^{\infty }L_{x}^{r}}\lesssim
\left( 1+t\right) ^{-\frac{3}{2}\left( \frac{1}{p}-\frac{1}{r}\right) }\left
\Vert u_{0L}\right \Vert _{L_{x}^{p}L_{\xi }^{2}}+e^{-\frac{t}{c}}\left
\Vert u_{0L}\right \Vert _{L_{\xi ,\beta }^{\infty }L_{x}^{2}}\text{,}
\end{equation*}%
\begin{equation*}
\left \Vert \varpi u_{L}\right \Vert _{L_{\xi ,\beta }^{\infty
}L_{x}^{r}}\lesssim \left( 1+t\right) ^{-\frac{3}{2}\left( \frac{1}{p}-\frac{
1}{r}\right) } \left \Vert u_{0L}\right \Vert _{L_{x}^{p}L_{\xi }^{2}} + e^{-%
\frac{t}{c}} \left \Vert \varpi u_{0L}\right \Vert _{L_{\xi ,\beta }^{\infty
}L_{x}^{2}} \text{,}
\end{equation*}%
for $\beta >3/2$, $2\leq r\leq \infty $ and $1\leq p\leq r$. Moreover, if $%
\mathrm{P}_{0}u_{0}=0$, then we will get extra $(1+t)^{-1/2}$ decay rate in
the above estimates.
\end{corollary}

On the other hand, in view of $\left( \ref{Eqn-short}\right) $,
\begin{equation*}
u_{S}=\mathbb{S}^{t}u_{0S}+\int_{0}^{t}\mathbb{S}^{t-\tau }Ku_{S}\left( \tau
\right) d\tau \text{.}
\end{equation*}%
Applying similar argument as those for $u$, together with $\left( \ref%
{short-r-out}\right) $, we get the $L_{x}^{r}$ estimate of the short wave $%
u_{S}$ for $2\leq r\leq \infty $ as well.

\begin{corollary}
\label{lemm-short}Let $u$ be a solution of $\left( \ref{Linearied Bolt}%
\right) $ with initial data $u_{0}$. Then%
\begin{equation*}
\left \Vert u_{S}\left( t\right) \right \Vert _{L_{\xi ,\beta }^{\infty
}L_{x}^{r}}\lesssim e^{-\frac{t}{c}}\left( \left \Vert u_{0S}\right \Vert
_{L_{\xi ,\beta }^{\infty }L_{x}^{r}}+\left \Vert u_{0}\right \Vert _{L_{\xi
}^{2}L_{x}^{2}}+\left \Vert u_{0}\right \Vert _{L_{\xi
}^{2}L_{x}^{2}}^{2/r}\left \Vert u_{0}\right \Vert _{L_{\xi
}^{2}L_{x}^{\infty }}^{1-2/r}\right) \text{,}
\end{equation*}%
\begin{equation*}
\left \Vert \varpi u_{S}\left( t\right) \right \Vert _{L_{\xi ,\beta
}^{\infty }L_{x}^{r}}\lesssim e^{-\frac{t}{c}}\left( \left \Vert \varpi
u_{0S}\right \Vert _{L_{\xi ,\beta }^{\infty }L_{x}^{r}}+\left \Vert
u_{0}\right \Vert _{L_{\xi }^{2}L_{x}^{2}}+\left \Vert u_{0}\right \Vert
_{L_{\xi }^{2}L_{x}^{2}}^{2/r}\left \Vert u_{0}\right \Vert _{L_{\xi
}^{2}L_{x}^{\infty }}^{1-2/r}\right) \text{,}
\end{equation*}%
for $\beta \geq 0$, $2\leq r\leq \infty $ and $1\leq p\leq r$.
\end{corollary}

\subsection{Estimates for the solution around $\mathcal{M}_{b}\label{Rmk for
fb}$}

In this section we mention that the foregoing results stated in Subsections %
\ref{collision} and \ref{refine} are also valid for the solution around $%
\mathcal{M}_{b}$, up to constants depending on $\rho $, $\lambda $ and $\mu $%
. To see this, write $F^{b}\left( t,x,\xi \right) =\sigma \widetilde{F}%
\left( \widetilde{t},\widetilde{x},\widetilde{\xi }\right) $ by the change
of variables
\begin{equation*}
\sigma =\rho \lambda ^{-3/2}\text{, }\widetilde{x}=x-\mu t\text{, }%
\widetilde{t}=\sqrt{\lambda }t\text{, }\widetilde{\xi }=\frac{\xi -\mu }{%
\sqrt{\lambda }}\text{,}
\end{equation*}%
and we discover
\begin{equation*}
\left \{
\begin{array}{l}
\partial _{\widetilde{t}}\widetilde{F}+\widetilde{\xi }\cdot \nabla _{%
\widetilde{x}}\widetilde{F}=\rho \sqrt{\lambda }^{(\gamma -1)}Q\left(
\widetilde{F},\widetilde{F}\right) =\widetilde{Q}\left( \widetilde{F},%
\widetilde{F}\right)
\vspace {3mm}
\\
\widetilde{F}\left( 0,\widetilde{x},\widetilde{\xi }\right) =\mathcal{M}_{a}(%
\widetilde{\xi })+\sqrt{\mathcal{M}_{a}(\widetilde{\xi })}\widetilde{f_{0}}%
\left( \widetilde{x},\widetilde{\xi }\right) \text{,}%
\end{array}%
\right.
\end{equation*}%
where $\widetilde{f_{0}}$ $\left( \widetilde{x},\widetilde{\xi }\right) =%
\frac{1}{\sqrt{\sigma }}f_{0}^{b}\left( \widetilde{x},\mu +\sqrt{\lambda }%
\widetilde{\xi }\right) $.

\section{Proof of the Main Theorem}

\label{sec:proof-main}

Now we turn to equation $\left( \ref{g-eqn1}\right) $. In view of (\ref%
{g-eqn1}), $g$ can be expressed by
\begin{equation}
g=\int_{0}^{t}\mathbb{G}^{t-s}\left \{ 2\Gamma \left( \frac{\sqrt{\mathcal{M}%
_{b}}}{\sqrt{\mathcal{M}_{a}}}f^{b},\frac{\mathcal{M}_{b}-\mathcal{M}_{a}}{%
\sqrt{\mathcal{M}_{a}}}\right) +\left[ 2\Gamma \left( f^{a},g\right) +\Gamma
\left( g,g\right) \right] \right \} ds\equiv \chi _{1}+\chi _{2}\text{.}
\label{X1X2}
\end{equation}%
Precisely,
\begin{equation*}
\chi _{1}=2\int_{0}^{t}\mathbb{G}^{t-s}\Gamma \left( \frac{\sqrt{\mathcal{M}%
_{b}}}{\sqrt{\mathcal{M}_{a}}}f^{b}(s),\frac{\mathcal{M}_{b}-\mathcal{M}_{a}%
}{\sqrt{\mathcal{M}_{a}}}\right) ds
\end{equation*}%
and
\begin{equation*}
\chi _{2}=\int_{0}^{t}\mathbb{G}^{t-s}\left[ 2\Gamma \left( f^{a},g\right)
(s)+\Gamma \left( g,g\right) (s)\right] ds\text{.}
\end{equation*}%
Keep in mind that $\mathrm{P%
}_{0}\Gamma \left( h_{1},h_{2}\right) =0$ during the course of the proof,
and it will give extra $\left( 1+t\right) ^{-1/2}$ time decay rate under the
operator $\mathbb{G}^{t}$.

Furthermore, we write $g=g_{1}+g_{2}$, where $g_{1}$ and $g_{2}$ solve the
equations
\begin{equation*}
\left \{
\begin{array}{l}
\displaystyle \partial _{t}g_{1}+\xi \cdot \nabla _{x}g_{1}+\nu (\xi
)g_{1}=2\Gamma \left( \frac{\sqrt{\mathcal{M}_{b}}}{\sqrt{\mathcal{M}_{a}}}%
f^{b},\frac{\mathcal{M}_{b}-\mathcal{M}_{a}}{\sqrt{\mathcal{M}_{a}}}\right) +%
\left[ 2\Gamma \left( f^{a},g\right) +\Gamma \left( g,g\right) \right] \text{%
,} \\[4mm]
\displaystyle g_{1}(0,x,\xi )=0\text{,}%
\end{array}%
\right.
\end{equation*}%
and
\begin{equation*}
\left \{
\begin{array}{l}
\displaystyle \partial _{t}g_{2}+\xi \cdot \nabla _{x}g_{2}+\nu (\xi
)g_{2}=Kg\text{,}\, \\[4mm]
\displaystyle g_{2}(0,x,\xi )=0\text{,}%
\end{array}%
\right.
\end{equation*}%
respectively. By the Duhamel principle, they can be written in terms of the
damped transport operator $\mathbb{S}^{t}$ as
\begin{equation*}
g_{1}(t,x,\xi )=\int_{0}^{t}\mathbb{S}^{t-s}\left[ 2\Gamma \left( \frac{%
\sqrt{\mathcal{M}_{b}}}{\sqrt{\mathcal{M}_{a}}}f^{b},\frac{\mathcal{M}_{b}-%
\mathcal{M}_{a}}{\sqrt{\mathcal{M}_{a}}}\right) +2\Gamma (f^{a},g)+\Gamma
(g,g)\right] ds\text{,}
\end{equation*}%
and
\begin{equation*}
g_{2}(t,x,\xi )=\int_{0}^{t}\mathbb{S}^{t-s}Kg(s)ds\text{.}
\end{equation*}%
In what follows we will estimate $\left \Vert g\right \Vert _{L_{\xi ,\beta
}^{\infty }L_{x}^{2}}$ and $\left \Vert g\right \Vert _{L_{\xi ,\beta
}^{\infty }L_{x}^{\infty }}$. The estimates will be expected to involve the
quantity $ \left( \mathcal{M}_{b}-\mathcal{M}_{a}\right) /\sqrt{%
\mathcal{M}_{a}}$. Before
proceeding, we show that this quantity can be simply dominated by the
parameters $\rho $, $\mu $ and $\lambda $.

\begin{lemma}
\label{Maxwellian-Esti} Let $1<\underline{\la}<\lambda <\overline{\lambda }<2$, $0<\rho <%
\overline{\rho }$ and $\beta \geq 0$. There exits a constant $C>0$ such that
\begin{equation*}
\lvert \frac{\mathcal{M}_{b}-\mathcal{M}_{b}}{\sqrt{\mathcal{M}_{a}}}\langle
\xi \rangle ^{\beta }\rvert \leq Ce^{-\frac{2-\overline{\la}}{16}|\xi|^{2}}\left( \lvert \rho -1\rvert +\lvert
\lambda -1\rvert +\lvert \mu \rvert \right) \text{,}
\end{equation*}%
where $\mathcal{M}_{a}$ and $\mathcal{M}_{b}$ are given as (\ref%
{Maxwellian-assu1}).
\end{lemma}

For the brevity of presentation, we denote the difference of macroscopic
quantities by
\begin{equation}
\mathcal{B}=\left \vert \rho -1\right \vert +\left \vert \lambda -1\right
\vert +\left \vert \mu \right \vert \text{.}  \label{eq:err}
\end{equation}

\begin{proof}
Define
\begin{equation*}
\mathcal{M}(\theta )=\frac{1+(\rho -1)\theta }{\left( 2\pi (1+(\lambda
-1)\theta )\right) ^{3/2}}e^{-\frac{\lvert \xi -\theta \mu \rvert ^{2}}{%
2(1+(\lambda -1)\theta )}}\text{.}
\end{equation*}%
Then by Mean Value Theorem
\begin{align*}
\mathcal{M}_{b}-\mathcal{M}_{a}& =\mathcal{M}(1)-\mathcal{M}(0) \\
& =\left. \mathcal{M}(\theta )\left[ -\frac{3(\lambda -1)}{2(1+\theta
(\lambda -1))}+\frac{\mu \cdot (\xi -\theta \mu )}{1+\theta (\lambda -1)}+%
\frac{(\lambda -1)\left \vert \xi -\theta \mu \right \vert ^{2}}{2(1+\theta
(\lambda -1))^{2}}+\frac{\rho -1}{1+\theta (\rho -1)}\right] \right \vert
_{\theta =\theta _{0}}
\end{align*}%
for some $0<\theta _{0}<1$. Then we have
\begin{align*}
\lvert \frac{\mathcal{M}_{b}-\mathcal{M}_{a}}{\sqrt{\mathcal{M}_{a}}}\rvert
& \leq Ce^{-\frac{\lvert \xi -\theta _{0}\mu \rvert ^{2}}{2(1+(\lambda
-1)\theta )}+\frac{\lvert \xi \rvert ^{2}}{4}}\left[ \lvert \lambda -1\rvert
+\lvert \mu \rvert \lvert \xi -\theta _{0}\mu \rvert +\lvert \lambda
-1\rvert \lvert \xi -\theta _{0}\mu \rvert ^{2}+\lvert \rho -1\rvert \right]
\\
& \leq C_{\delta }e^{-\frac{\lvert \xi -\theta _{0}\mu \rvert ^{2}}{2\lambda
(1+\delta )}+\frac{\lvert \xi \rvert ^{2}}{4}}\left[ \lvert \rho -1\rvert
+\lvert \mu \rvert +\lvert \lambda -1\rvert \right] \text{,}
\end{align*}%
where the polynomial $\lvert \xi -\theta _{0}\mu \rvert $ is absorbed by the
exponential function and $\delta $ can be any positive number.

Note that
\begin{equation*}
e^{-\frac{\lvert \xi -\theta _{0}\mu \rvert ^{2}}{2\lambda (1+\delta )}+%
\frac{\lvert \xi \rvert ^{2}}{4}}=e^{-\frac{2-\lambda (1+\delta )}{%
4(1+\delta )\lambda }\lvert \xi \rvert ^{2}+\frac{\theta _{0}\mu \cdot \xi }{%
(1+\delta )\lambda }-\frac{\theta _{0}^{2}\lvert \mu \rvert ^{2}}{2(1+\delta
)\lambda }}\text{.}
\end{equation*}%
In view of $1<\underline{\la}<\lambda <\overline{\lambda }<2$, one can choose $\delta >0$
sufficiently small to ensure the coefficient in front of $\lvert \xi \rvert
^{2}$ is negative. It then follows that any polynomial $\langle \xi \rangle
^{\beta }$ with $\beta \geq 0$ can be absorbed by the exponential function.
This completes the proof.
\end{proof}

\subsection{$L_{x}^{2}$ Estimate of $g$}

Let $T>0$ be any finite number and $\beta >3/2+2\gamma $. Then for $0\leq
t\leq T$,

\begin{align}
& \left \langle \xi \right \rangle ^{\beta }|g_{1}|_{L_{x}^{2}}  \notag \\
& \lesssim \int_{0}^{t}e^{-\nu (\xi )(t-s)}\nu (\xi )\left \Vert \nu
^{-1}\Gamma \left( \frac{\sqrt{\mathcal{M}_{b}}}{\sqrt{\mathcal{M}_{a}}}%
f^{b},\frac{\mathcal{M}_{b}-\mathcal{M}_{a}}{\sqrt{\mathcal{M}_{a}}}\right)
\right \Vert _{L_{\xi ,\beta }^{\infty }L_{x}^{2}}(s)ds  \notag \\
& \quad +\int_{0}^{t}e^{-\nu (\xi )(t-s)}\nu (\xi )\left[ \Vert \nu
^{-1}\Gamma (f^{a},g)\Vert _{L_{\xi ,\beta }^{\infty }L_{x}^{2}}(s)+\Vert
\nu ^{-1}\Gamma (g,g)\Vert _{L_{\xi ,\beta }^{\infty }L_{x}^{2}}(s)\right] ds
\label{inho-g1-1} \\
& \lesssim \mathcal{B}\int_{0}^{t}e^{-\nu (\xi )(t-s)}\nu (\xi )\left \Vert
\frac{\sqrt{\mathcal{M}_{b}}}{\sqrt{\mathcal{M}_{a}}}f^{b}\right \Vert
_{L_{\xi ,\beta }^{\infty }L_{x}^{2}}(s)ds  \notag \\
& \quad +\int_{0}^{t}e^{-\nu (\xi )(t-s)}\nu (\xi )\left[ \Vert f^{a}\Vert
_{L_{\xi ,\beta }^{\infty }L_{x}^{\infty }}\Vert g\Vert _{L_{\xi ,\beta
}^{\infty }L_{x}^{2}}(s)+\Vert g\Vert _{L_{\xi ,\beta }^{\infty
}L_{x}^{\infty }}\Vert g\Vert _{L_{\xi ,\beta }^{\infty }L_{x}^{2}}(s)\right]
ds  \notag \\
& \lesssim \varepsilon \mathcal{B}\int_{0}^{t}e^{-\nu (\xi )(t-s)}\nu (\xi
)(1+s)^{-\frac{3}{4}}ds\left \Vert \frac{\sqrt{\mathcal{M}_{b}}}{\sqrt{%
\mathcal{M}_{a}}}f_{0}^{b}\right \Vert _{L_{\xi ,\beta }^{\infty
}(L_{x}^{\infty }\cap L_{x}^{1})}  \notag \\
& \quad +\varepsilon \int_{0}^{t}e^{-\nu (\xi )(t-s)}\nu (\xi )(1+s)^{-\frac{%
3}{2}-\frac{1}{4}}ds\sup_{0\leq s\leq T}(1+s)^{1/4}\Vert g\Vert _{L_{\xi
,\beta }^{\infty }L_{x}^{2}}  \notag \\
& \quad +\int_{0}^{t}e^{-\nu (\xi )(t-s)}\nu (\xi )(1+s)^{-\frac{1}{4}%
-1}ds\sup_{0\leq s\leq T}(1+s)^{1/4}\Vert g\Vert _{L_{\xi ,\beta }^{\infty
}L_{x}^{2}}\cdot \sup_{0\leq s\leq T}(1+s)\Vert g\Vert _{L_{\xi ,\beta
}^{\infty }L_{x}^{\infty }}  \notag \\
& \lesssim \varepsilon \mathcal{B}(1+t)^{-3/4}+\varepsilon
(1+t)^{-7/4}\sup_{0\leq s\leq T}(1+s)^{1/4}\Vert g\Vert _{L_{\xi ,\beta
}^{\infty }L_{x}^{2}}  \notag \\
& \quad +(1+t)^{-5/4}\sup_{0\leq s\leq T}(1+s)^{1/4}\Vert g\Vert _{L_{\xi
,\beta }^{\infty }L_{x}^{2}}\cdot \sup_{0\leq s\leq T}(1+s)\Vert g\Vert
_{L_{\xi ,\beta }^{\infty }L_{x}^{\infty }}\, \text{,}  \notag
\end{align}%
by Theorem \ref{prop: nonlinear}, Lemma \ref{Maxwellian-Esti} and the
remarks in Section \ref{Rmk for fb}. Here $\mathcal{B}$ is the difference of
macroscopic quantities defined in \eqref{eq:err}.
\begin{align*}
\left \langle \xi \right \rangle ^{\beta }|g_{2}|_{L_{x}^{2}}& \lesssim
\int_{0}^{t}e^{-\nu (\xi )(t-s)}\nu (\xi )\Vert \nu (\xi )^{-1}Kg\Vert
_{L_{\xi ,\beta }^{\infty }L_{x}^{2}}\left( s\right) ds \\
& \lesssim \int_{0}^{t}e^{-\nu (\xi )(t-s)}\nu (\xi )\Vert g\Vert _{L_{\xi
,\beta -\gamma }^{\infty }L_{x}^{2}}\left( s\right) ds \\
& \lesssim (1+t)^{-1/4}\sup_{0\leq s\leq T}(1+s)^{1/4}\Vert g\Vert _{L_{\xi
,\beta -\gamma }^{\infty }L_{x}^{2}}\text{.}
\end{align*}

To obtain $\sup_{0\leq s\leq T}(1+s)^{1/4}\Vert g\Vert _{L_{\xi ,\beta
-\gamma }^{\infty }L_{x}^{2}}$, we now estimate $\Vert \chi _{1}\Vert
_{L_{\xi ,\beta -\gamma }^{\infty }L_{x}^{2}}$ and $\Vert \chi _{2}\Vert
_{L_{\xi ,\beta -\gamma }^{\infty }L_{x}^{2}}$ instead. In view of Theorems %
\ref{Thm for linear}, \ref{prop: nonlinear} and Lemma \ref{Maxwellian-Esti},
\begin{eqnarray*}
&&\left \Vert \chi _{1}\right \Vert _{L_{\xi ,\beta -\gamma }^{\infty
}L_{x}^{2}} \\
&\lesssim &\int_{0}^{t}\left( 1+t-s\right) ^{-1/2}\left \Vert 2\left \langle
\xi \right \rangle ^{\beta -\gamma }\Gamma \left( \frac{\sqrt{\mathcal{M}_{b}%
}}{\sqrt{\mathcal{M}_{a}}}f^{b},\frac{\mathcal{M}_{b}-\mathcal{M}_{a}}{\sqrt{%
\mathcal{M}_{a}}}\right) \right \Vert _{L_{\xi }^{\infty }L_{x}^{2}}\left(s\right) ds \\
&\lesssim &\mathcal{B}\int_{0}^{t}\left( 1+t-s\right) ^{-1/2}\left \Vert
\left \langle \xi \right \rangle ^{\beta }\frac{\sqrt{\mathcal{M}_{b}}}{%
\sqrt{\mathcal{M}_{a}}}f^{b}\right \Vert _{L_{\xi }^{\infty }L_{x}^{2}}\left(s\right) ds \\
&\lesssim &\varepsilon \mathcal{B}\int_{0}^{t}\left( 1+t-s\right)
^{-1/2}(1+s)^{-3/4}ds \\
&\lesssim &\varepsilon \mathcal{B}(1+t)^{-1/4}\text{,}
\end{eqnarray*}%
and
\begin{align*}
& \quad \Vert \chi _{2}\Vert _{L_{\xi ,\beta -\gamma }^{\infty }L_{x}^{2}} \\
& \lesssim \int_{0}^{t}(1+t-s)^{-5/4}\left( \Vert \Gamma (f^{a},g)\Vert
_{L_{\xi ,\beta -\gamma }^{\infty }L_{x}^{2}}\left( s\right) +\Vert \Gamma
(f^{a},g)\Vert _{L_{\xi ,\beta -\gamma }^{\infty }L_{x}^{1}}\left( s\right)
\right) ds \\
& \quad +\int_{0}^{t}(1+t-s)^{-5/4}\left( \Vert \Gamma (g,g)\Vert _{L_{\xi
,\beta -\gamma }^{\infty }L_{x}^{2}}\left( s\right) +\Vert \Gamma (g,g)\Vert
_{L_{\xi ,\beta -\gamma }^{\infty }L_{x}^{1}}\left( s\right) \right) ds \\
& \lesssim \varepsilon \int_{0}^{t}(1+t-s)^{-5/4}(1+s)^{-1}ds\sup_{0\leq
s\leq T}(1+s)^{1/4}\Vert g\Vert _{L_{\xi ,\beta }^{\infty }L_{x}^{2}} \\
& \quad +\int_{0}^{t}(1+t-s)^{-5/4}(1+s)^{-5/4}ds\left( \sup_{0\leq s\leq
T}(1+s)^{1/4}\Vert g\Vert _{L_{\xi ,\beta }^{\infty }L_{x}^{2}}\sup_{0\leq
s\leq T}(1+s)\Vert g\Vert _{L_{\xi ,\beta }^{\infty }L_{x}^{\infty }}\right)
\\
& \quad +\int_{0}^{t}(1+t-s)^{-5/4}(1+s)^{-1/2}ds\left( \sup_{0\leq s\leq
T}(1+s)^{1/4}\Vert g\Vert _{L_{\xi ,\beta }^{\infty }L_{x}^{2}}\right) ^{2}
\end{align*}%
which implies that
\begin{align*}
& \quad \Vert \chi _{2}\Vert _{L_{\xi ,\beta -\gamma }^{\infty }L_{x}^{2}} \\
& \lesssim \varepsilon (1+t)^{-1}\sup_{0\leq s\leq T}(1+s)^{1/4}\Vert g\Vert
_{L_{\xi ,\beta }^{\infty }L_{x}^{2}} \\
& \quad +(1+t)^{-1/2}\left( \sup_{0\leq s\leq T}(1+s)^{1/4}\Vert g\Vert
_{L_{\xi ,\beta }^{\infty }L_{x}^{2}}\right) ^{2}+(1+t)^{-5/4}\text{%
\thinspace }\left( \sup_{0\leq s\leq T}(1+s)\Vert g\Vert _{L_{\xi ,\beta
}^{\infty }L_{x}^{\infty }}\right) ^{2}\text{.}
\end{align*}%
Therefore, we obtain
\begin{align}
(1+t)^{1/4}\Vert g_{2}\Vert _{L_{\xi ,\beta }^{\infty }L_{x}^{2}}& \lesssim
\varepsilon \mathcal{B}+\varepsilon \sup_{0\leq s\leq T}(1+s)^{1/4}\Vert
g\Vert _{L_{\xi ,\beta }^{\infty }L_{x}^{2}}  \label{inho-g2-1} \\
& \quad +\left( \sup_{0\leq s\leq T}(1+s)^{1/4}\Vert g\Vert _{L_{\xi ,\beta
}^{\infty }L_{x}^{2}}\right) ^{2}+\left( \sup_{0\leq s\leq T}(1+s)\Vert
g\Vert _{L_{\xi ,\beta }^{\infty }L_{x}^{\infty }}\right) ^{2}\text{.}
\notag
\end{align}%
\textbf{Summing up above estimates, we have the $L_{\xi ,\beta }^{\infty
}L_{x}^{2}$ estimate of $g$: }
\begin{align}
(1+t)^{1/4}\Vert g\Vert _{L_{\xi ,\beta }^{\infty }L_{x}^{2}}& \lesssim
\varepsilon \mathcal{B}+\varepsilon \sup_{0\leq s\leq t}(1+s)^{1/4}\Vert
g\Vert _{L_{\xi ,\beta }^{\infty }L_{x}^{2}}  \notag \\
& \quad +\left( \sup_{0\leq s\leq t}(1+s)^{1/4}\Vert g\Vert _{L_{\xi ,\beta
}^{\infty }L_{x}^{2}}\right) ^{2}+\left( \sup_{0\leq s\leq t}(1+s)\Vert
g\Vert _{L_{\xi ,\beta }^{\infty }L_{x}^{\infty }}\right) ^{2}\text{,}
\label{inho-g-1}
\end{align}%
for $0\leq t\leq T$.

\subsection{$L^{\infty}_{x}$ Estimate on $g$}

Applying similar argument as those for the $L_{x}^{2}$ estimate of $g$, we
have for $0\leq t\leq T$ and $\beta >3/2+2\gamma $
\begin{align}
(1+t)\Vert g_{1}\Vert _{L_{\xi ,\beta }^{\infty }L_{x}^{\infty }}& \lesssim
(1+t)^{-1/2}\varepsilon \mathcal{B}+\varepsilon (1+t)^{-3/2}\sup_{0\leq
s\leq T}(1+s)\Vert g\Vert _{L_{\xi ,\beta }^{\infty }L_{x}^{\infty }}  \notag
\\
& \quad +(1+t)^{-1}\left( \sup_{0\leq s\leq T}(1+s)\Vert g\Vert _{L_{\xi
,\beta }^{\infty }L_{x}^{\infty }}\right) ^{2}  \label{inho-g1-2}
\end{align}%
and
\begin{equation*}
\left \langle \xi \right \rangle ^{\beta }|g_{2}|_{L_{x}^{\infty }}\lesssim
(1+t)^{-1}\sup_{0\leq s\leq T}(1+s)\Vert g\Vert _{L_{\xi ,\beta -\gamma
}^{\infty }L_{x}^{\infty }}\text{.}
\end{equation*}%
As same as the $L_{x}^{2}$ case, , we will obtain $\sup_{0\leq s\leq
T}(1+s)\Vert g\Vert _{L_{\xi ,\beta -\gamma }^{\infty }L_{x}^{\infty }}$ by
estimating $\chi _{1}$ and $\chi _{2}$. For $\chi_{1}$, we have
\begin{equation*}
\left \Vert \chi _{1}\right \Vert _{L_{\xi ,\beta -\gamma }^{\infty
}L_{x}^{\infty }}\lesssim \varepsilon \left( 1+t\right) ^{-1}\mathcal{B}%
\text{.}
\end{equation*}
We need some refined estimate for the linearized Boltzmann equation in
Section \ref{refine} to prove this estimate. Since the proof is delicate and
lengthy, we postpone the detail to the next subsection (Proposition \ref%
{Prop-X1}).

For $\chi_{2}$, by Theorems \ref{prop: nonlinear} and \ref{Thm for linear}
we have%
\begin{eqnarray*}
\left \Vert \chi _{2}\right \Vert _{L_{\xi ,\beta -\gamma }^{\infty
}L_{x}^{\infty }} &\lesssim &\int_{0}^{t}\left( 1+t-s\right) ^{-5/4}\left(
\left \Vert \Gamma \left( f^{a},g\right) \right \Vert _{L_{\xi ,\beta
-\gamma }^{\infty }L_{x}^{\infty }}\left( s\right) +\left \Vert \Gamma
\left( f^{a},g\right) \right \Vert _{L_{\xi ,\beta -\gamma }^{\infty
}L_{x}^{2}}\left( s\right) \right) ds \\
&&+\int_{0}^{t}\left( 1+t-s\right) ^{-5/4}\left( \left \Vert \Gamma \left(
g,g\right) \right \Vert _{L_{\xi ,\beta -\gamma }^{\infty }L_{x}^{\infty
}}\left( s\right) +\left \Vert \Gamma \left( g,g\right) \right \Vert
_{L_{\xi ,\beta -\gamma }^{\infty }L_{x}^{2}}\left( s\right) \right) ds \\
&\lesssim &\varepsilon \left( 1+t\right) ^{-5/4}\left( \sup_{0\leq s\leq
T}\left( 1+s\right) \Vert g\Vert _{L_{\xi ,\beta }^{\infty }L_{x}^{\infty
}}+\sup_{0\leq s\leq T}\left( 1+s\right) ^{1/4}\Vert g\Vert _{L_{\xi ,\beta
}^{\infty }L_{x}^{2}}\right) \\
&&+\left( 1+t\right) ^{-5/4}\left[ \left( \sup_{0\leq s\leq T}\left(
1+s\right) ^{1/4}\Vert g\Vert _{L_{\xi ,\beta }^{\infty }L_{x}^{2}}\right)
^{2}+\left( \sup_{0\leq s\leq T}\left( 1+s\right) \Vert g\Vert _{L_{\xi
,\beta }^{\infty }L_{x}^{\infty }}\right) ^{2}\right] \text{.}
\end{eqnarray*}

Therefore,
\begin{eqnarray}
(1+t)\Vert g_{2}\Vert _{L_{\xi ,\beta }^{\infty }L_{x}^{\infty }} &\lesssim
&\varepsilon \mathcal{B}+\varepsilon \left( \sup_{0\leq s\leq T}\left(
1+s\right) \Vert g\Vert _{L_{\xi ,\beta }^{\infty }L_{x}^{\infty
}}+\sup_{0\leq s\leq T}\left( 1+s\right) ^{1/4}\Vert g\Vert _{L_{\xi ,\beta
}^{\infty }L_{x}^{2}}\right)  \label{inho-g2-2} \\
&&+\left( \sup_{0\leq s\leq T}\left( 1+s\right) ^{1/4}\Vert g\Vert _{L_{\xi
,\beta }^{\infty }L_{x}^{2}}\right) ^{2}+\left( \sup_{0\leq s\leq T}\left(
1+s\right) \Vert g\Vert _{L_{\xi ,\beta }^{\infty }L_{x}^{\infty }}\right)
^{2}\text{.}  \notag
\end{eqnarray}%
Combining $\left( \ref{inho-g1-2}\right) $ and $\left( \ref{inho-g2-2}%
\right) $, \textbf{we conclude the $L_{\xi ,\beta }^{\infty }L_{x}^{\infty }$
estimate of $g$:}
\begin{eqnarray}
(1+t)\Vert g\Vert _{L_{\xi ,\beta }^{\infty }L_{x}^{\infty }} &\lesssim
&\varepsilon \mathcal{B}+\varepsilon \left( \sup_{0\leq s\leq T}\left(
1+s\right) \Vert g\Vert _{L_{\xi ,\beta }^{\infty }L_{x}^{\infty
}}+\sup_{0\leq s\leq T}\left( 1+s\right) ^{1/4}\Vert g\Vert _{L_{\xi ,\beta
}^{\infty }L_{x}^{2}}\right)  \label{inho-g-2} \\
&&+\left( \sup_{0\leq s\leq T}\left( 1+s\right) ^{1/4}\Vert g\Vert _{L_{\xi
,\beta }^{\infty }L_{x}^{2}}\right) ^{2}+\left( \sup_{0\leq s\leq T}\left(
1+s\right) \Vert g\Vert _{L_{\xi ,\beta }^{\infty }L_{x}^{\infty }}\right)
^{2}\text{.}  \notag
\end{eqnarray}

Now let
\begin{equation*}
Q\left( T\right) =\underset{0\leq s\leq T}{\sup }\left( (1+s)\left \Vert
\left \langle \xi \right \rangle ^{\beta }g/\mathcal{B}\right \Vert _{L_{\xi
}^{\infty }L_{x}^{\infty }}\text{ }+(1+s)^{1/4}\left \Vert \left \langle \xi
\right \rangle ^{\beta }g/\mathcal{B}\right \Vert _{L_{\xi }^{\infty
}L_{x}^{2}}\right) \text{,}
\end{equation*}%
for any finite $T>0$. From $\left( \ref{inho-g-1}\right) $ and $\left( \ref%
{inho-g-2}\right) $ it gives the inequality
\begin{equation*}
Q\left( T\right) \leq C_{1}\varepsilon +C_{2}\varepsilon Q\left( T\right) +C%
\mathcal{B}Q^{2}\left( T\right)
\end{equation*}%
for any finite $T\geq 0$. Since $g\left( 0,x,\xi \right) =0$, we get $%
Q\left( T\right) \leq \widetilde{C}\varepsilon $ for some constant $%
\widetilde{C}>0$ and for all $T\geq 0$ whenever $\varepsilon >0$ is
sufficiently small. The proof is completed.$%
\hfill%
\square $

\subsection{$L_{\protect \xi,\protect \beta -\protect \gamma}^{\infty
}L_{x}^{\infty}$ estimate of $\protect \chi_{1}$}

In this section we are devoted to the estimate of $\left \Vert \chi
_{1}\right \Vert _{L_{\xi ,\beta -\gamma }^{\infty }L_{x}^{\infty }}$, $%
\beta >3/2+2\gamma $.

\begin{remark}
\label{rmk:slow} At the first sight, it readily follows from Theorems \ref%
{prop: nonlinear} and \ref{Thm for linear} that
\begin{equation*}
\left \Vert \chi _{1}\right \Vert _{L_{\xi ,\beta -\gamma }^{\infty
}L_{x}^{\infty }}\lesssim \varepsilon \mathcal{B}\int_{0}^{t}\left(
1+t-s\right) ^{-5/4}\left( 1+s\right) ^{-3/4}ds\lesssim \varepsilon \mathcal{%
B}\left( 1+t\right) ^{-3/4}\text{.}
\end{equation*}%
However, this estimate for $\chi _{1}$ is not satisfactory and we can
improve it up to the time decay rate $\left( 1+t\right) ^{-1}$ by the
refined estimate for the linearized Boltzmann equation.
\end{remark}

\begin{proposition}
\label{Prop-X1}Let $\beta >3/2+2\gamma $. Then $\chi _{1}$ defined as $%
\left( \ref{X1X2}\right) $ satisfies%
\begin{equation*}
\left \Vert \chi _{1}\right \Vert _{L_{\xi ,\beta -\gamma }^{\infty
}L_{x}^{\infty }}\leq C\varepsilon \left( 1+t\right) ^{-1}\mathcal{B}\left(
\left \Vert \frac{\sqrt{\mathcal{M}_{b}}}{\sqrt{\mathcal{M}_{a}}}%
f_{0}^{b}\right \Vert _{L_{\xi ,\beta }^{\infty }\left( L_{x}^{\infty }\cap
L_{x}^{1}\right) }+\left \Vert f_{0}^{b}\right \Vert _{L_{x}^{1}L_{\xi
}^{2}}\right)
\end{equation*}%
for some constant $C>0$ independent of $t$.
\end{proposition}

To prove this, we represent $f^{b}$ by
\begin{equation*}
f^{b}=\varepsilon \mathbb{G}_{b}^{t}f_{0}^{b}+\int_{0}^{t}\mathbb{G}%
_{b}^{t-\tau }\Gamma _{b}\left( f^{b},f^{b}\right)\left(\tau\right) d\tau
\end{equation*}
and then write $\chi _{1}$ as
\begin{eqnarray*}
\chi _{1} &=&\int_{0}^{t}\mathbb{G}^{t-\tau }\Gamma \left( \frac{\sqrt{%
\mathcal{M}_{b}}}{\sqrt{\mathcal{M}_{a}}}f^{b}\,,\frac{\mathcal{M}_{b}-%
\mathcal{M}_{a}}{\sqrt{\mathcal{M}_{a}}}\right)\left(\tau\right) d\tau \\
&=&\varepsilon \int_{0}^{t}\mathbb{G}^{t-\tau }\Gamma \left( \frac{\sqrt{%
\mathcal{M}_{b}}}{\sqrt{\mathcal{M}_{a}}}\mathbb{G}_{b}^{\tau }f_{0}^{b}\,,%
\frac{\mathcal{M}_{b}-\mathcal{M}_{a}}{\sqrt{\mathcal{M}_{a}}}\right)\left(\tau\right) d\tau
\\
&&+\int_{0}^{t}\mathbb{G}^{t-\tau }\Gamma \left( \frac{\sqrt{\mathcal{M}_{b}}%
}{\sqrt{\mathcal{M}_{a}}}\int_{0}^{\tau }\mathbb{G}_{b}^{\tau -s}\Gamma
_{b}\left( f^{b},f^{b}\right) ds\,,\frac{\mathcal{M}_{b}-\mathcal{M}_{a}}{%
\sqrt{\mathcal{M}_{a}}}\right)\left(\tau\right) d\tau \\
&\equiv &\chi _{11}+\chi _{12}\text{.}
\end{eqnarray*}%
We further split $\chi _{11}$ into two parts
\begin{equation*}
\chi _{111}=\varepsilon \int_{0}^{\frac{t}{2}}\mathbb{G}^{t-\tau }\Gamma
\left( \frac{\sqrt{\mathcal{M}_{b}}}{\sqrt{\mathcal{M}_{a}}}\mathbb{G}%
_{b}^{\tau }f_{0}^{b}\,,\frac{\mathcal{M}_{b}-\mathcal{M}_{a}}{\sqrt{%
\mathcal{M}_{a}}}\right) \left(\tau\right)d\tau \text{,}\,
\end{equation*}%
and%
\begin{equation*}
\chi _{112}=\varepsilon \int_{\frac{t}{2}}^{t}\mathbb{G}^{t-\tau }\Gamma
\left( \frac{\sqrt{\mathcal{M}_{b}}}{\sqrt{\mathcal{M}_{a}}}\mathbb{G}%
_{b}^{\tau }f_{0}^{b}\,,\frac{\mathcal{M}_{b}-\mathcal{M}_{a}}{\sqrt{%
\mathcal{M}_{a}}}\right) \left(\tau\right)d\tau \text{.}
\end{equation*}%
Among them, $\chi _{112}$ is delicate, so we first deal with this term and
use the long-short wave decomposition to split the integrand into six parts
\begin{equation*}
\chi _{112}=\varepsilon \int_{\frac{t}{2}}^{t}\left( \mathbb{G}%
_{L,0}^{t-\tau }+\mathbb{G}_{L,\bot }^{t-\tau }+\mathbb{G}_{S}^{t-\tau
}\right) \Gamma \left( \frac{\sqrt{\mathcal{M}_{b}}}{\sqrt{\mathcal{M}_{a}}}%
\left( \mathbb{G}_{b,L}^{\tau }+\mathbb{G}_{b,S}^{\tau }\right) f_{0}^{b}\,,%
\frac{\mathcal{M}_{b}-\mathcal{M}_{a}}{\sqrt{\mathcal{M}_{a}}}\right) \left(\tau\right)d\tau
\text{.}
\end{equation*}%
For the purpose of simplification, we denote
\begin{equation*}
\mathcal{T}h:=\Gamma \left( \frac{\sqrt{\mathcal{M}_{b}}}{\sqrt{\mathcal{M}%
_{a}}}h\,,\frac{\mathcal{M}_{b}-\mathcal{M}_{a}}{\sqrt{\mathcal{M}_{a}}}%
\right) \text{.}
\end{equation*}%
We will estimate them term by term.

In view of Proposition \ref{Prop-LS}, Corollary \ref{lemm-long}, and the
fact that $\mathrm{P}_{0}\mathcal{T}\mathbb{G}_{b,L}^{\tau }f_{0}^{b}=0$,
\begin{eqnarray*}
&&\left \Vert \mathbb{G}_{L,0}^{t-\tau }\mathcal{T}\mathbb{G}_{b,L}^{\tau
}f_{0}^{b}\right \Vert _{L_{x}^{\infty }L_{\xi }^{2}} \\
&\lesssim &\left( 1+t-\tau \right) ^{-\frac{3}{2p}-\frac{1}{2}}\left \Vert
\mathcal{T}\mathbb{G}_{b,L}^{\tau }f_{0}^{b}\right \Vert _{L_{x}^{p}L_{\xi
}^{2}} \\
&\lesssim &\left( 1+t-\tau \right) ^{-\frac{3}{2p}-\frac{1}{2}}\left \Vert
\Gamma \left( \frac{\sqrt{\mathcal{M}_{b}}}{\sqrt{\mathcal{M}_{a}}}\mathbb{G}%
_{b,L}^{\tau }f_{0}^{b}\,,\frac{\mathcal{M}_{b}-\mathcal{M}_{a}}{\sqrt{%
\mathcal{M}_{a}}}\right) \right \Vert _{L_{x}^{p}L_{\xi }^{2}} \\
&\lesssim &\left( 1+t-\tau \right) ^{-\frac{3}{2p}-\frac{1}{2}}\mathcal{B}%
\left \Vert \nu \left( \xi \right) \frac{\sqrt{\mathcal{M}_{b}}}{\sqrt{%
\mathcal{M}_{a}}}\mathbb{G}_{b,L}^{\tau }f_{0}^{b}\right \Vert
_{L_{x}^{p}L_{\xi }^{2}} \\
&\lesssim &\left( 1+t-\tau \right) ^{-\frac{3}{2p}-\frac{1}{2}}\mathcal{B}%
\left \Vert \frac{\sqrt{\mathcal{M}_{b}}}{\sqrt{\mathcal{M}_{a}}}\mathbb{G}%
_{b,L}^{\tau }f_{0}^{b}\right \Vert _{L_{\xi ,\beta }^{\infty }L_{x}^{p}} \\
&\lesssim &\left( 1+t-\tau \right) ^{-\frac{3}{2p}-\frac{1}{2}}\left( 1+\tau
\right) ^{-\frac{3}{2}\left( \frac{1}{r}-\frac{1}{p}\right) }\mathcal{B}%
\left( \left \Vert \frac{\sqrt{\mathcal{M}_{b}}}{\sqrt{\mathcal{M}_{a}}}%
f_{0L}^{b}\right \Vert _{L_{\xi ,\beta }^{\infty }L_{x}^{2}}+\left \Vert
f_{0L}^{b}\right \Vert _{L_{x}^{r}L_{\xi }^{2}}\right) \text{,}
\end{eqnarray*}%
for $2\leq p\leq \infty $ and $1\leq r\leq p$. Taking $r=1$ and $p>3$ gives%
\begin{equation}
\int_{\frac{t}{2}}^{t}\left \Vert \mathbb{G}_{L,0}^{t-\tau }\mathcal{T}%
\mathbb{G}_{b,L}^{\tau }f_{0}^{b}\right \Vert _{L_{x}^{\infty }L_{\xi
}^{2}}d\tau \lesssim \left( 1+t\right) ^{-1}\mathcal{B}\left( \left \Vert
\frac{\sqrt{\mathcal{M}_{b}}}{\sqrt{\mathcal{M}_{a}}}f_{0L}^{b}\right \Vert
_{L_{\xi ,\beta }^{\infty }L_{x}^{2}}+\left \Vert f_{0L}^{b}\right \Vert
_{L_{x}^{1}L_{\xi }^{2}}\right) \text{.}  \label{GL-0-L}
\end{equation}

By the Sobolev inequality,
\begin{eqnarray*}
\left \Vert \mathbb{G}_{L,\bot }^{t-\tau }\mathcal{T}\mathbb{G}_{b,L}^{\tau
}f_{0}^{b}\right \Vert _{L_{x}^{\infty }L_{\xi }^{2}} &\lesssim &\left \Vert
\nabla _{x}^{2}\mathbb{G}_{L,\bot }^{t-\tau }\mathcal{T}\mathbb{G}%
_{b,L}^{\tau }f_{0}^{b}\right \Vert _{L_{x}^{2}L_{\xi }^{2}}^{1/2}\cdot
\left \Vert \nabla _{x}\mathbb{G}_{L,\bot }^{t-\tau }\mathcal{T}\mathbb{G}%
_{b,L}^{\tau }f_{0}^{b}\right \Vert _{L_{x}^{2}L_{\xi }^{2}}^{1/2} \\
&\lesssim &\left \Vert \mathbb{G}_{L,\bot }^{t-\tau }\mathcal{T}\nabla
_{x}^{2}\mathbb{G}_{b,L}^{\tau }f_{0}^{b}\right \Vert _{L_{x}^{2}L_{\xi
}^{2}}^{1/2}\cdot \left \Vert \mathbb{G}_{L,\bot }^{t-\tau }\mathcal{T}%
\nabla _{x}\mathbb{G}_{b,L}^{\tau }f_{0}^{b}\right \Vert _{L_{x}^{2}L_{\xi
}^{2}}^{1/2}\text{.}
\end{eqnarray*}%
In view of Proposition \ref{Prop-LS} and Corollary \ref{lemm-long},%
\begin{eqnarray*}
\left \Vert \mathbb{G}_{L,\bot }^{t-\tau }\mathcal{T}\nabla _{x}^{2}\mathbb{G%
}_{b,L}^{\tau }f_{0}^{b}\right \Vert _{L_{x}^{2}L_{\xi }^{2}} &\lesssim &e^{-%
\frac{t-\tau }{c}}\left \Vert \mathcal{T}\nabla _{x}^{2}\mathbb{G}%
_{b,L}^{\tau }f_{0}^{b}\right \Vert _{L_{x}^{2}L_{\xi }^{2}} \\
&\lesssim &e^{-\frac{t-\tau }{c}}\left \Vert \Gamma \left( \frac{\sqrt{%
\mathcal{M}_{b}}}{\sqrt{\mathcal{M}_{a}}}\nabla _{x}^{2}\mathbb{G}%
_{b,L}^{\tau }f_{0}^{b}\,,\frac{\mathcal{M}_{b}-\mathcal{M}_{a}}{\sqrt{%
\mathcal{M}_{a}}}\right) \right \Vert _{L_{x}^{2}L_{\xi }^{2}} \\
&\lesssim &e^{-\frac{t-\tau }{c}}\mathcal{B}\left \Vert \nu \left( \xi
\right) \frac{\sqrt{\mathcal{M}_{b}}}{\sqrt{\mathcal{M}_{a}}}\nabla _{x}^{2}%
\mathbb{G}_{b,L}^{\tau }f_{0}^{b}\right \Vert _{L_{x}^{2}L_{\xi }^{2}} \\
&\lesssim &e^{-\frac{t-\tau }{c}}\mathcal{B}\left \Vert \frac{\sqrt{\mathcal{%
M}_{b}}}{\sqrt{\mathcal{M}_{a}}}\nabla _{x}^{2}\mathbb{G}_{b,L}^{\tau
}f_{0}^{b}\right \Vert _{L_{\xi ,\beta }^{\infty }L_{x}^{2}} \\
&\lesssim &e^{-\frac{t-\tau }{c}}\left( 1+\tau \right) ^{-\frac{3}{4}-\frac{2%
}{2}}\mathcal{B}\left( \left \Vert \frac{\sqrt{\mathcal{M}_{b}}}{\sqrt{%
\mathcal{M}_{a}}}f_{0L}^{b}\right \Vert _{L_{\xi ,\beta }^{\infty
}L_{x}^{2}}+\left \Vert f_{0L}^{b}\right \Vert _{L_{x}^{1}L_{\xi
}^{2}}\right) \text{.}
\end{eqnarray*}%
Similarly,
\begin{equation*}
\left \Vert \mathbb{G}_{L,\bot }^{t-\tau }\mathcal{T}\nabla _{x}\mathbb{G}%
_{b,L}^{\tau }f_{0}^{b}\right \Vert _{L_{x}^{2}L_{\xi }^{2}}\lesssim e^{-%
\frac{t-\tau }{c}}\left( 1+\tau \right) ^{-\frac{3}{4}-\frac{1}{2}}\mathcal{B%
}\left( \left \Vert \frac{\sqrt{\mathcal{M}_{b}}}{\sqrt{\mathcal{M}_{a}}}%
f_{0L}^{b}\right \Vert _{L_{\xi ,\beta }^{\infty }L_{x}^{2}}+\left \Vert
f_{0L}^{b}\right \Vert _{L_{x}^{1}L_{\xi }^{2}}\right) \text{.}
\end{equation*}%
Therefore,%
\begin{equation*}
\left \Vert \mathbb{G}_{L,\bot }^{t-\tau }\mathcal{T}\mathbb{G}_{b,L}^{\tau
}f_{0}^{b}\right \Vert _{L_{x}^{\infty }L_{\xi }^{2}}\lesssim e^{-\frac{%
t-\tau }{c}}\left( 1+\tau \right) ^{-\frac{3}{2}}\mathcal{B}\left( \left
\Vert \frac{\sqrt{\mathcal{M}_{b}}}{\sqrt{\mathcal{M}_{a}}}f_{0L}^{b}\right
\Vert _{L_{\xi ,\beta }^{\infty }L_{x}^{2}}+\left \Vert f_{0L}^{b}\right
\Vert _{L_{x}^{1}L_{\xi }^{2}}\right)
\end{equation*}%
and
\begin{equation}
\quad \left \Vert \int_{\frac{t}{2}}^{t}\mathbb{G}_{L,\bot }^{t-\tau }%
\mathcal{T}\mathbb{G}_{b,L}^{\tau }f_{0}^{b}d\tau \right \Vert
_{L_{x}^{\infty }L_{\xi }^{2}}\lesssim \left( 1+t\right) ^{-\frac{3}{2}}%
\mathcal{B}\left( \left \Vert \frac{\sqrt{\mathcal{M}_{b}}}{\sqrt{\mathcal{M}%
_{a}}}f_{0L}^{b}\right \Vert _{L_{\xi ,\beta }^{\infty }L_{x}^{2}}+\left
\Vert f_{0L}^{b}\right \Vert _{L_{x}^{1}L_{\xi }^{2}}\right) \text{.}
\label{GL-per-L}
\end{equation}

By Proposition \ref{Prop-LS}, Corollary \ref{lemm-short} and the fact that $%
\mathrm{P}_{0}\mathcal{T}\mathbb{G}_{b,S}^{\tau }f_{0}^{b}=0$,%
\begin{eqnarray}
\int_{\frac{t}{2}}^{t}\left \Vert \left( \mathbb{G}_{L,0}^{t-\tau }+\mathbb{G%
}_{L,\bot }^{t-\tau }\right) \mathcal{T}\mathbb{G}_{b,S}^{\tau
}f_{0}^{b}\right \Vert _{L_{x}^{\infty }L_{\xi }^{2}}d\tau &\lesssim &\int_{%
\frac{t}{2}}^{t}\left( 1+t-\tau \right) ^{-\frac{5}{4}}\left \Vert \mathcal{T%
}\mathbb{G}_{b,S}^{\tau }f_{0}^{b}\right \Vert _{L_{x}^{2}L_{\xi }^{2}}d\tau
\label{GLS} \\
&\lesssim &\mathcal{B}\int_{\frac{t}{2}}^{t}\left( 1+t-\tau \right) ^{-\frac{%
5}{4}}\left \Vert \nu \left( \xi \right) \frac{\sqrt{\mathcal{M}_{b}}}{\sqrt{%
\mathcal{M}_{a}}}\mathbb{G}_{S}^{\tau }f_{0}^{b}\right \Vert _{L_{\xi
}^{2}L_{x}^{2}}d\tau  \notag \\
&\lesssim &\int_{\frac{t}{2}}^{t}\mathcal{B}\left( 1+t-\tau \right) ^{-\frac{%
5}{4}}\left \Vert \frac{\sqrt{\mathcal{M}_{b}}}{\sqrt{\mathcal{M}_{a}}}%
\mathbb{G}_{S}^{\tau }f_{0}^{b}\right \Vert _{L_{\xi ,\beta }^{\infty
}L_{x}^{2}}d\tau  \notag \\
&\lesssim &e^{-\frac{t}{2c}}\mathcal{B}\left( \left \Vert \frac{\sqrt{%
\mathcal{M}_{b}}}{\sqrt{\mathcal{M}_{a}}}f_{0S}^{b}\right \Vert _{L_{\xi
,\beta }^{\infty }L_{x}^{2}}+\left \Vert f_{0}^{b}\right \Vert _{L_{\xi
}^{2}L_{x}^{2}}\right) \text{.}  \notag
\end{eqnarray}

In accordance with $\left( \ref{short-r-out}\right) $ and Corollary \ref%
{lemm-short},
\begin{eqnarray}
&&\int_{\frac{t}{2}}^{t}\left \Vert \mathbb{G}_{S}^{t-\tau }\mathcal{T}%
\mathbb{G}_{S}^{\tau }f_{0}^{b}\right \Vert _{L_{x}^{\infty }L_{\xi
}^{2}}d\tau  \label{GSS} \\
&\lesssim &\int_{\frac{t}{2}}^{t}e^{-\frac{t-\tau }{c}}\left[ \left \Vert
\mathcal{T}\mathbb{G}_{S}^{\tau }f_{0}^{b}\right \Vert _{L_{\xi
}^{2}L_{x}^{2}}+\left \Vert \mathcal{T}\mathbb{G}_{S}^{\tau }f_{0}^{b}\right
\Vert _{L_{\xi }^{2}L_{x}^{\infty }}\right] d\tau  \notag \\
&\lesssim &\int_{\frac{t}{2}}^{t}e^{-\frac{t-\tau }{c}}\mathcal{B}\left(
\left \Vert \nu \left( \xi \right) \frac{\sqrt{\mathcal{M}_{b}}}{\sqrt{%
\mathcal{M}_{a}}}\mathbb{G}_{S}^{\tau }f_{0}^{b}\right \Vert _{L_{\xi
}^{2}L_{x}^{2}}+\left \Vert \nu \left( \xi \right) \frac{\sqrt{\mathcal{M}%
_{b}}}{\sqrt{\mathcal{M}_{a}}}\mathbb{G}_{S}^{\tau }f_{0}^{b}\right \Vert
_{L_{\xi }^{2}L_{x}^{\infty }}\right) d\tau  \notag \\
&\lesssim &\mathcal{B}\int_{\frac{t}{2}}^{t}e^{-\frac{t-\tau }{c}}\left(
\left \Vert \frac{\sqrt{\mathcal{M}_{b}}}{\sqrt{\mathcal{M}_{a}}}\mathbb{G}%
_{S}^{\tau }f_{0}^{b}\right \Vert _{L_{\xi ,\beta }^{\infty
}L_{x}^{2}}+\left \Vert \frac{\sqrt{\mathcal{M}_{b}}}{\sqrt{\mathcal{M}_{a}}}%
\mathbb{G}_{S}^{\tau }f_{0}^{b}\right \Vert _{L_{\xi ,\beta }^{\infty
}L_{x}^{\infty }}\right) d\tau  \notag \\
&\lesssim &\mathcal{B}\int_{\frac{t}{2}}^{t}e^{-\frac{t-\tau }{c}}e^{-\frac{%
\tau }{c}}d\tau \left( \left \Vert \frac{\sqrt{\mathcal{M}_{b}}}{\sqrt{%
\mathcal{M}_{a}}}f_{0S}^{b}\right \Vert _{L_{\xi ,\beta }^{\infty }\left(
L_{x}^{2}\cap L_{x}^{\infty }\right) }+\left \Vert f_{0}^{b}\right \Vert
_{L_{\xi }^{2}\left( L_{x}^{2}\cap L_{x}^{\infty }\right) }\right)  \notag \\
&\lesssim &e^{-\frac{t}{2c}}\mathcal{B}\left( \left \Vert \frac{\sqrt{%
\mathcal{M}_{b}}}{\sqrt{\mathcal{M}_{a}}}f_{0S}^{b}\right \Vert _{L_{\xi
,\beta }^{\infty }\left( L_{x}^{2}\cap L_{x}^{\infty }\right) }+\left \Vert
f_{0}^{b}\right \Vert _{L_{\xi }^{2}\left( L_{x}^{2}\cap L_{x}^{\infty
}\right) }\right) \text{.}  \notag
\end{eqnarray}

By the Sobolev inequality, Proposition \ref{Prop-LS} and Corollary \ref%
{lemm-long},
\begin{align*}
& \quad \left \Vert \mathbb{G}_{S}^{t-\tau }\mathcal{T}\mathbb{G}_{L}^{\tau
}f_{0}^{b}\right \Vert _{L_{x}^{\infty }L_{\xi }^{2}} \\
& \lesssim \left \Vert \nabla _{x}^{2}\mathbb{G}_{S}^{t-\tau }\mathcal{T}%
\mathbb{G}_{L}^{\tau }f_{0}^{b}\right \Vert _{L_{x}^{2}L_{\xi
}^{2}}^{1/2}\cdot \left \Vert \nabla _{x}\mathbb{G}_{S}^{t-\tau }\mathcal{T}%
\mathbb{G}_{L}^{\tau }f_{0}^{b}\right \Vert _{L_{x}^{2}L_{\xi }^{2}}^{1/2} \\
& \lesssim e^{-\frac{t-\tau }{c}}\left \Vert \mathcal{T}\nabla _{x}^{2}%
\mathbb{G}_{L}^{\tau }f_{0}^{b}\right \Vert _{L_{\xi
}^{2}L_{x}^{2}}^{1/2}\left \Vert \mathcal{T}\nabla _{x}\mathbb{G}_{L}^{\tau
}f_{0}^{b}\right \Vert _{L_{\xi }^{2}L_{x}^{2}}^{1/2} \\
& \lesssim e^{-\frac{t-\tau }{c}}\left( 1+\tau \right) ^{\frac{1}{2}(-\frac{3%
}{4}-\frac{2}{2})+\frac{1}{2}(-\frac{3}{4}-\frac{1}{2})}\mathcal{B}\left(
\left \Vert \frac{\sqrt{\mathcal{M}_{b}}}{\sqrt{\mathcal{M}_{a}}}%
f_{0L}^{b}\right \Vert _{L_{\xi ,\beta }^{\infty }L_{x}^{2}}+\left \Vert
f_{0L}^{b}\right \Vert _{L_{x}^{1}L_{\xi }^{2}}\right)
\end{align*}%
and so

\begin{equation}
\int_{\frac{t}{2}}^{t}\left \Vert \mathbb{G}_{S}^{t-\tau }\mathcal{T}\mathbb{%
G}_{L}^{\tau }f_{0}^{b}\right \Vert _{L_{x}^{\infty }L_{\xi }^{2}}d\tau
\lesssim \mathcal{B}(1+t)^{-3/2}\left( \left \Vert \frac{\sqrt{\mathcal{M}%
_{b}}}{\sqrt{\mathcal{M}_{a}}}f_{0L}^{b}\right \Vert _{L_{\xi ,\beta
}^{\infty }L_{x}^{2}}+\left \Vert f_{0L}^{b}\right \Vert _{L_{x}^{1}L_{\xi
}^{2}}\right) \text{.}  \label{GSL}
\end{equation}

Gathering $\left( \ref{GL-0-L}\right) $-$\left( \ref{GSL}\right) $, we obtain%
\begin{eqnarray*}
&&\left \Vert \chi _{112}\right \Vert _{L_{x}^{\infty }L_{\xi }^{2}} \\
&\lesssim &\varepsilon \left( 1+t\right) ^{-1}\mathcal{B}\left( \left \Vert
\frac{\sqrt{\mathcal{M}_{b}}}{\sqrt{\mathcal{M}_{a}}}f_{0L}^{b}\right \Vert
_{L_{\xi ,\beta }^{\infty }L_{x}^{2}}+\left \Vert f_{0L}^{b}\right \Vert
_{L_{x}^{1}L_{\xi }^{2}}+\left \Vert \frac{\sqrt{\mathcal{M}_{b}}}{\sqrt{%
\mathcal{M}_{a}}}f_{0S}^{b}\right \Vert _{L_{\xi ,\beta }^{\infty }\left(
L_{x}^{2}\cap L_{x}^{\infty }\right) }+\left \Vert f_{0}^{b}\right \Vert
_{L_{\xi }^{2}\left( L_{x}^{2}\cap L_{x}^{\infty }\right) }\right) \\
&\lesssim &\varepsilon \left( 1+t\right) ^{-1}\mathcal{B}\left( \left \Vert
\frac{\sqrt{\mathcal{M}_{b}}}{\sqrt{\mathcal{M}_{a}}}f_{0}^{b}\right \Vert
_{L_{\xi ,\beta }^{\infty }\left( L_{x}^{2}\cap L_{x}^{\infty }\right)
}+\left \Vert f_{0}^{b}\right \Vert _{L_{x}^{1}L_{\xi }^{2}}\right) \text{,}
\end{eqnarray*}%
since $\left \Vert f_{0L}^{b}\right \Vert _{L_{x}^{1}L_{\xi }^{2}}\lesssim
\left \Vert f_{0}^{b}\right \Vert _{L_{x}^{1}L_{\xi }^{2}}\,$and%
\begin{equation*}
\left \Vert \frac{\sqrt{\mathcal{M}_{b}}}{\sqrt{\mathcal{M}_{a}}}%
f_{0S}^{b}\right \Vert _{L_{\xi ,\beta }^{\infty }L_{x}^{\infty }}=\left
\Vert \frac{\sqrt{\mathcal{M}_{b}}}{\sqrt{\mathcal{M}_{a}}}\left(
f_{0}^{b}-f_{0L}^{b}\right) \right \Vert _{L_{\xi ,\beta }^{\infty
}L_{x}^{\infty }}\leq \left \Vert \frac{\sqrt{\mathcal{M}_{b}}}{\sqrt{%
\mathcal{M}_{a}}}f_{0}^{b}\right \Vert _{L_{\xi ,\beta }^{\infty
}L_{x}^{\infty }}+\left \Vert \frac{\sqrt{\mathcal{M}_{b}}}{\sqrt{\mathcal{M}%
_{a}}}f_{0L}^{b}\right \Vert _{L_{\xi ,\beta }^{\infty }L_{x}^{2}}\text{.}
\end{equation*}

Next we see $\left \Vert \chi _{111}\right \Vert _{L_{x}^{\infty }L_{\xi
}^{2}} $. In light of Theorem \ref{Thm for linear} and Proposition \ref{Prop-linear} and the fact that $%
\mathrm{P}_{0}\Gamma \left( \frac{\sqrt{\mathcal{M}_{b}}}{\sqrt{\mathcal{M}%
_{a}}}\mathbb{G}_{b}^{\tau }f_{0}^{b}\,,\frac{\mathcal{M}_{b}-\mathcal{M}_{a}%
}{\sqrt{\mathcal{M}_{a}}}\right) =0$,

\begin{eqnarray*}
\left \Vert \chi _{111}\right \Vert _{L_{x}^{\infty }L_{\xi }^{2}}
&=&\varepsilon \left \Vert \int_{0}^{\frac{t}{2}}\mathbb{G}^{t-\tau }\Gamma
\left( \frac{\sqrt{\mathcal{M}_{b}}}{\sqrt{\mathcal{M}_{a}}}\mathbb{G}%
_{b}^{\tau }f_{0}^{b}\,,\frac{\mathcal{M}_{b}-\mathcal{M}_{a}}{\sqrt{%
\mathcal{M}_{a}}}\right) d\tau \right \Vert _{L_{x}^{\infty }L_{\xi }^{2}} \\
&\lesssim &\varepsilon \mathcal{B}\int_{0}^{\frac{t}{2}}\left( 1+t-\tau
\right) ^{-\frac{5}{4}}\left( \left \Vert \nu \left( \xi \right) \frac{\sqrt{%
\mathcal{M}_{b}}}{\sqrt{\mathcal{M}_{a}}}\mathbb{G}_{b}^{\tau
}f_{0}^{b}\right \Vert _{L_{x}^{2}L_{\xi }^{2}}+\left \Vert \nu \left( \xi
\right) \frac{\sqrt{\mathcal{M}_{b}}}{\sqrt{\mathcal{M}_{a}}}\mathbb{G}%
_{b}^{\tau }f_{0}^{b}\right \Vert _{L_{\xi }^{2}L_{x}^{\infty }}\right) d\tau
\\
&\lesssim &\varepsilon \mathcal{B}\int_{0}^{\frac{t}{2}}\left( 1+t-\tau
\right) ^{-\frac{5}{4}}\left( \left \Vert \frac{\sqrt{\mathcal{M}_{b}}}{%
\sqrt{\mathcal{M}_{a}}}\mathbb{G}_{b}^{\tau }f_{0}^{b}\right \Vert _{L_{\xi
,\beta }^{\infty }L_{x}^{2}}+\left \Vert \frac{\sqrt{\mathcal{M}_{b}}}{\sqrt{%
\mathcal{M}_{a}}}\mathbb{G}_{b}^{\tau }f_{0}^{b}\right \Vert _{L_{\xi ,\beta
}^{\infty }L_{x}^{\infty }}\right) d\tau \\
&\lesssim &\varepsilon \mathcal{B}\int_{0}^{\frac{t}{2}}\left( 1+t-\tau
\right) ^{-\frac{5}{4}}\left( 1+\tau \right) ^{-\frac{3}{4}}d\tau \left(
\left \Vert \frac{\sqrt{\mathcal{M}_{b}}}{\sqrt{\mathcal{M}_{a}}}%
f_{0}^{b}\right \Vert _{L_{\xi ,\beta }^{\infty }L_{x}^{2}}+\left \Vert
\frac{\sqrt{\mathcal{M}_{b}}}{\sqrt{\mathcal{M}_{a}}}f_{0}^{b}\right \Vert
_{L_{\xi ,\beta }^{\infty }L_{x}^{\infty }}+\left \Vert f_{0}^{b}\right
\Vert _{L_{x}^{1}L_{\xi }^{2}}\right) \\
&\lesssim &\varepsilon \left( 1+t\right) ^{-1}\mathcal{B}\left( \left \Vert
\frac{\sqrt{\mathcal{M}_{b}}}{\sqrt{\mathcal{M}_{a}}}f_{0}^{b}\right \Vert
_{L_{\xi ,\beta }^{\infty }\left( L_{x}^{2}\cap L_{x}^{\infty }\right)
}+\left \Vert f_{0}^{b}\right \Vert _{L_{x}^{1}L_{\xi }^{2}}\right) \text{.}
\end{eqnarray*}%
Therefore,
\begin{equation}
\left \Vert \chi _{11}\right \Vert _{L_{x}^{\infty }L_{\xi }^{2}}=\left
\Vert \chi _{111}+\chi _{112}\right \Vert _{L_{x}^{\infty }L_{\xi
}^{2}}\lesssim \varepsilon \left( 1+t\right) ^{-1}\mathcal{B}\left( \left
\Vert \frac{\sqrt{\mathcal{M}_{b}}}{\sqrt{\mathcal{M}_{a}}}f_{0}^{b}\right
\Vert _{L_{\xi ,\beta }^{\infty }\left( L_{x}^{\infty }\cap L_{x}^{2}\right)
}+\left \Vert f_{0}^{b}\right \Vert _{L_{x}^{1}L_{\xi }^{2}}\right) \text{.}
\label{X11-Estimate}
\end{equation}%
Note that $\chi _{11}$ can be expressed by
\begin{equation*}
\chi _{11}=\int_{0}^{t}\mathbb{S}^{t-\tau }\left[ K\chi _{11}\left( \tau
\right) +\varepsilon \Gamma \left( \frac{\sqrt{\mathcal{M}_{b}}}{\sqrt{%
\mathcal{M}_{a}}}\mathbb{G}_{b}^{\tau }f_{0}^{b},\frac{\mathcal{M}_{b}-%
\mathcal{M}_{a}}{\sqrt{\mathcal{M}_{a}}}\right) \right] d\tau \text{.}
\end{equation*}%
By Proposition \ref{Prop-linear} and $\left( \ref{X11-Estimate}\right) $, we
have%
\begin{eqnarray*}
\left \Vert \chi _{11}\right \Vert _{L_{x}^{\infty }L_{\xi }^{\infty }}
&\lesssim &\int_{0}^{t}e^{-\frac{\left( t-\tau \right) }{c}}\left( \left
\Vert \chi _{11}\left( \tau \right) \right \Vert _{L_{x}^{\infty }L_{\xi
}^{2}}+\varepsilon \mathcal{B}\left \Vert \nu \left( \xi \right) \frac{\sqrt{%
\mathcal{M}_{b}}}{\sqrt{\mathcal{M}_{a}}}\mathbb{G}_{b}^{\tau
}f_{0}^{b}\right \Vert _{L_{x}^{\infty }L_{\xi }^{\infty }}\right) d\tau \\
&\lesssim &\varepsilon \left( 1+t\right) ^{-1}\mathcal{B}\left( \left \Vert
\frac{\sqrt{\mathcal{M}_{b}}}{\sqrt{\mathcal{M}_{a}}}f_{0}^{b}\right \Vert
_{L_{\xi ,\beta }^{\infty }\left( L_{x}^{\infty }\cap L_{x}^{2}\right)
}+\left \Vert f_{0}^{b}\right \Vert _{L_{x}^{1}L_{\xi }^{2}}\right) \text{,}%
\,
\end{eqnarray*}%
and then
\begin{equation}
\left \Vert \chi _{11}\right \Vert _{L_{x}^{\infty }L_{\xi ,\beta -\gamma
}^{\infty }}\lesssim \varepsilon \left( 1+t\right) ^{-1}\mathcal{B}\left(
\left \Vert \frac{\sqrt{\mathcal{M}_{b}}}{\sqrt{\mathcal{M}_{a}}}%
f_{0}^{b}\right \Vert _{L_{\xi ,\beta }^{\infty }\left( L_{x}^{\infty }\cap
L_{x}^{2}\right) }+\left \Vert f_{0}^{b}\right \Vert _{L_{x}^{1}L_{\xi
}^{2}}\right)  \label{X11-weighted}
\end{equation}%
via the bootstrap argument.

Finally, we consider $\left \Vert \chi _{12}\right \Vert _{L_{\xi ,\beta
-\gamma }^{\infty }L_{x}^{\infty }}$. Owing to Theorems \ref{prop: nonlinear}
and \ref{Thm for linear}, we have
\begin{eqnarray*}
&&\left \Vert \chi _{12}\right \Vert _{L_{\xi ,\beta -2\gamma }^{\infty
}L_{x}^{\infty }} \\
&\lesssim &\int_{0}^{t}\left( 1+t-\tau \right) ^{-\frac{5}{4}}\left \Vert
\Gamma \left( \frac{\sqrt{\mathcal{M}_{b}}}{\sqrt{\mathcal{M}_{a}}}%
\int_{0}^{\tau }\mathbb{G}_{b}^{\tau -s}\Gamma _{b}\left( f^{b},f^{b}\right)
ds,\frac{\mathcal{M}_{b}-\mathcal{M}_{a}}{\sqrt{\mathcal{M}_{a}}}\right)
\right \Vert _{L_{\xi ,\beta -2\gamma }^{\infty }\left( L_{x}^{2}\cap
L_{x}^{\infty }\right) }d\tau \\
&\lesssim &\mathcal{B}\int_{0}^{t}\left( 1+t-\tau \right) ^{-\frac{5}{4}%
}\int_{0}^{\tau }\left \Vert \frac{\sqrt{\mathcal{M}_{b}}}{\sqrt{\mathcal{M}%
_{a}}}\mathbb{G}_{b}^{\tau -s}\Gamma _{b}\left( f^{b},f^{b}\right) \right
\Vert _{L_{\xi ,\beta -\gamma }^{\infty }\left( L_{x}^{2}\cap L_{x}^{\infty
}\right) }dsd\tau \\
&\lesssim &\mathcal{B}\int_{0}^{t}\left( 1+t-\tau \right) ^{-\frac{5}{4}%
}\int_{0}^{\tau }\left( 1+\tau -s\right) ^{-\frac{5}{4}}\left \Vert \frac{%
\sqrt{\mathcal{M}_{b}}}{\sqrt{\mathcal{M}_{a}}}\Gamma _{b}\left(
f^{b},f^{b}\right) \right \Vert _{L_{\xi ,\beta -\gamma }^{\infty }\left(
L_{x}^{1}\cap L_{x}^{2}\cap L_{x}^{\infty }\right) }dsd\tau \\
&\lesssim &\varepsilon ^{2}\mathcal{B}\int_{0}^{t}\left( 1+t-\tau \right) ^{-%
\frac{5}{4}}\int_{0}^{\tau }\left( 1+\tau -s\right) ^{-\frac{5}{4}}\left(
1+s\right) ^{-\frac{3}{2}}dsd\tau \left \Vert \frac{\sqrt{\mathcal{M}_{b}}}{%
\sqrt{\mathcal{M}_{a}}}f_{0}^{b}\right \Vert _{L_{\xi ,\beta }^{\infty
}\left( L_{x}^{1}\cap L_{x}^{\infty }\right) }^{2} \\
&\lesssim &\varepsilon ^{2}\mathcal{B}\left( 1+t\right) ^{-\frac{5}{4}}\left
\Vert \frac{\sqrt{\mathcal{M}_{b}}}{\sqrt{\mathcal{M}_{a}}}f_{0}^{b}\right
\Vert _{L_{\xi ,\beta }^{\infty }\left( L_{x}^{1}\cap L_{x}^{\infty }\right)
}^{2}\text{.}\,
\end{eqnarray*}%
Note that
\begin{equation*}
\chi _{12}=\int_{0}^{t}\mathbb{S}^{t-\tau }\Gamma \left( \frac{\sqrt{%
\mathcal{M}_{b}}}{\sqrt{\mathcal{M}_{a}}}\int_{0}^{\tau }\mathbb{G}%
_{b}^{\tau -s}\Gamma _{b}\left( f^{b},f^{b}\right) ds\,,\frac{\mathcal{M}%
_{b}-\mathcal{M}_{a}}{\sqrt{\mathcal{M}_{a}}}\right) \left(\tau\right)d\tau +\int_{0}^{t}%
\mathbb{S}^{t-\tau}K\chi _{12}(\tau)d\tau\text{.}
\end{equation*}%
Hence in view of $\left( \ref{K-Lp}\right) $, Theorem \ref{prop: nonlinear}
and Proposition \ref{Prop-linear},
\begin{align*}
& \quad \left \vert \left \langle \xi \right \rangle ^{\beta -\gamma }\chi
_{12}\right \vert _{L_{x}^{\infty }} \\
& \lesssim \mathcal{B}\int_{0}^{t}e^{-\nu \left( \xi \right) \left( t-\tau
\right) }\nu \left( \xi \right) \int_{0}^{\tau }\left \Vert \frac{\sqrt{%
\mathcal{M}_{b}}}{\sqrt{\mathcal{M}_{a}}}\mathbb{G}_{b}^{\tau -s}\Gamma
_{b}\left( f^{b},f^{b}\right) \right \Vert _{L_{\xi ,\beta -\gamma }^{\infty
}L_{x}^{\infty }}dsd\tau \\
& \quad +\int_{0}^{t}e^{-\nu _{0}(t-\tau )}\left \Vert \chi _{12}\right
\Vert _{L_{\xi ,\beta -2\gamma }^{\infty }L_{x}^{\infty }}(\tau)d\tau \\
& \lesssim \varepsilon ^{2}\mathcal{B}\int_{0}^{t}e^{-\nu \left( \xi \right)
\left( t-\tau \right) }\nu \left( \xi \right) \int_{0}^{\tau }\left( 1+\tau
-s\right) ^{-5/4}\left( 1+s\right) ^{-9/4}dsd\tau \left \Vert \frac{\sqrt{%
\mathcal{M}_{b}}}{\sqrt{\mathcal{M}_{a}}}f_{0}^{b}\right \Vert _{L_{\xi
,\beta }^{\infty }\left( L_{x}^{1}\cap L_{x}^{\infty }\right) }^{2} \\
& \quad +\varepsilon ^{2}\mathcal{B}\int_{0}^{t}e^{-\nu _{0}(t-\tau
)}(1+\tau)^{-5/4}d\tau\left \Vert \frac{\sqrt{\mathcal{M}_{b}}}{\sqrt{\mathcal{M}%
_{a}}}f_{0}^{b}\right \Vert _{L_{\xi ,\beta }^{\infty }\left( L_{x}^{1}\cap
L_{x}^{\infty }\right) }^{2} \\
& \lesssim \varepsilon ^{2}\mathcal{B}(1+t)^{-5/4}\left \Vert \frac{\sqrt{%
\mathcal{M}_{b}}}{\sqrt{\mathcal{M}_{a}}}f_{0}^{b}\right \Vert _{L_{\xi
,\beta }^{\infty }\left( L_{x}^{1}\cap L_{x}^{\infty }\right) }^{2}\, \text{,%
}
\end{align*}%
so that
\begin{equation}
\left \Vert \chi _{12}\right \Vert _{L_{\xi ,\beta -\gamma }^{\infty
}L_{x}^{\infty }}\lesssim \varepsilon ^{2}\mathcal{B}(1+t)^{-5/4}\left \Vert
\frac{\sqrt{\mathcal{M}_{b}}}{\sqrt{\mathcal{M}_{a}}}f_{0}^{b}\right \Vert
_{L_{\xi ,\beta }^{\infty }\left( L_{x}^{1}\cap L_{x}^{\infty }\right) }^{2}%
\text{.}  \label{X12-weighted}
\end{equation}%
Combining $\left( \ref{X11-weighted}\right) $ and $\left( \ref{X12-weighted}%
\right) $, the proof of Proposition \ref{Prop-X1} is completed.$%
\hfill%
\square $

%
%
%
%
%
%
%
%
%
%
%
%
%
%
%
%
%
%
%
%
%
%
\appendix

\section{Heat equation}

\label{sec:heat}

\begin{theorem}
\label{heat}Let $\mu \in \mathbb{R}^{3}$, $1\leq \lambda \leq 2$. Assume
that $h^{a}$ and $h^{b}$ satisfy the heat equations in the whole space $%
\mathbb{R}^{3}$, i.e.,%
\begin{equation*}
\partial _{t}h^{a}=\Delta h^{a}\text{,}
\end{equation*}%
and
\begin{equation*}
\partial _{t}h^{b}+\mu \cdot \nabla h^{b}=\lambda ^{\left( 2-\gamma \right)
/2}\Delta h^{b}\text{,}
\end{equation*}%
with initial data $h_{0}^{a}=h_{0}^{b}=h_{0}\in L_{x}^{1}\left( \mathbb{R}%
^{3}\right) $. Then there exists a constant $C>0$ independent of time such
that
\begin{equation*}
\left \vert h^{b}-h^{a}\right \vert _{L_{x}^{\infty }}\leq
C(1+t)^{-1}\left \vert h_{0}\right \vert _{L_{x}^{1}}\left( \left \vert \lambda
-1\right \vert (1+t)^{-1/2}+\left \vert \mu \right \vert \right) \text{,}
\end{equation*}%
\begin{equation*}
\left \vert h^{b}-h^{a}\right \vert _{L_{x}^{2}}\leq C(1+t)^{-1/4}\left \vert
h_{0}\right \vert _{L_{x}^{1}}\left( \left \vert \lambda -1\right \vert
(1+t)^{-1/2}+\left \vert \mu \right \vert \right) \text{,}
\end{equation*}%
for $t\geq 1$.
\end{theorem}

To simplify the notation, we set $\kappa =\lambda ^{\left( 2-\gamma \right)
/2}$. We will provide two different methods to prove the theorem. The first
method is to study the difference of the two solutions. In view of the exact
solution formula associated with the heat kernel, we have
\begin{equation*}
h^{a}\left( t,x\right) =\int_{\mathbb{R}^{3}}\frac{1}{\left( 4\pi t\right)
^{3/2}}e^{-\frac{\left \vert x-y\right \vert ^{2}}{4t}}h_{0}\left( y\right)
dy\text{,}
\end{equation*}%
\begin{equation*}
h^{b}\left( t,x\right) =\int_{\mathbb{R}^{3}}\frac{1}{(4\pi \kappa t)^{3/2}}%
e^{-\frac{\left \vert x-\mu t-y\right \vert ^{2}}{4\kappa t}}h_{0}\left(
y\right) dy\text{,}
\end{equation*}%
so that
\begin{equation*}
h^{a}\left( t,x\right) -h^{b}\left( t,x\right) =\int_{\mathbb{R}^{3}}\left[
\frac{1}{\left( 4\pi t\right) ^{3/2}}e^{-\frac{\left \vert x-y\right \vert
^{2}}{4t}}-\frac{1}{(4\pi \kappa t)^{3/2}}e^{-\frac{\left \vert x-\mu
t-y\right \vert ^{2}}{4\kappa t}}\right] h_{0}\left( y\right) dy\text{.}
\end{equation*}

To proceed, we need the following lemma:

\begin{lemma}
\label{Lemm-H}For $\mu \in \mathbb{R}^{3}$, $1\leq \lambda \leq 2$, $t>0$,%
\begin{equation*}
\left \vert \frac{e^{-\frac{\left \vert x\right \vert ^{2}}{4t}}}{(4\pi
t)^{3/2}}-\frac{e^{-\frac{\left \vert x-\mu t\right \vert ^{2}}{4\ka t}}}{(4\pi \kappa t)^{3/2}}\right \vert _{L_{x}^{p}}\leq
Ct^{-\frac{3}{2}(1-\frac{1}{p})}\left[ \lvert \kappa -1\rvert +\lvert \mu
\rvert \sqrt{t}\right] \text{, }1\leq p\leq \infty \text{\thinspace ,}
\end{equation*}%
for some constant $C>0$ independent of $\lambda $, $u$, $p$.
\end{lemma}

\begin{proof}
By mean value theorem
\begin{multline*}
\frac{e^{-\frac{\left \vert x-\mu t\right \vert ^{2}}{4\kappa t}}}{(4\pi
\kappa t)^{3/2}}-\frac{e^{-\frac{\left \vert x\right \vert ^{2}}{4t}}}{(4\pi
t)^{3/2}}=\left[ \frac{e^{-\frac{\lvert x-\theta \mu t\rvert ^{2}}{%
4(1+(\kappa -1)\theta )t}}}{(4\pi (1+(\kappa -1)\theta )t)^{3/2}}\right]
_{\theta =0}^{\theta =1} \\
=\int_{0}^{1}\frac{e^{-\frac{\lvert x-\theta \mu t\rvert ^{2}}{4(1+(\kappa
-1)\theta )t}}}{(4\pi (1+(\kappa -1)\theta )t)^{3/2}}\left[ -\frac{3(\kappa
-1)}{2(\theta (\kappa -1)+1)}+\frac{\mu \cdot (x-\theta \mu t)}{2(\theta
(\kappa -1)+1)}+\frac{(\kappa -1)\left \vert x-\theta \mu t\right \vert ^{2}%
}{4t(\theta (\kappa -1)+1)^{2}}\right] d\theta \text{.}
\end{multline*}%
It then immediately follows that for $1\leq p\leq \infty $,
\begin{eqnarray*}
&&\left \vert \frac{e^{-\frac{\left \vert x-\mu t\right \vert ^{2}}{4\kappa t%
}}}{(4\pi \kappa t)^{3/2}}-\frac{e^{-\frac{\left \vert x\right \vert ^{2}}{4t%
}}}{(4\pi t)^{3/2}}\right \vert _{L_{x}^{p}} \\
&\leq &\int_{0}^{1}\left \vert \frac{e^{-\frac{\lvert x-\theta \mu t\rvert
^{2}}{4(1+(\kappa -1)\theta )t}}}{(4\pi (1+(\kappa -1)\theta )t)^{3/2}}\left[
-\frac{3(\kappa -1)}{2(\theta (\kappa -1)+1)}+\frac{\mu \cdot (x-\theta \mu
t)}{2(\theta (\kappa -1)+1)}+\frac{(\kappa -1)\left \vert x-\theta \mu
t\right \vert ^{2}}{4t(\theta (\kappa -1)+1)^{2}}\right] \right \vert
_{L_{x}^{p}}d\theta \\
&\lesssim &\int_{0}^{1}\left \vert \frac{e^{-\frac{\left \vert x-\theta \mu
t\right \vert ^{2}}{Ct}}}{t^{3/2}}\left[ \lvert \kappa -1\rvert +\lvert \mu
\rvert \sqrt{t}\right] \right \vert _{L_{x}^{p}}d\theta \lesssim t^{-\frac{3%
}{2}(1-\frac{1}{p})}\left[ \lvert \kappa -1\rvert +\lvert \mu \rvert \sqrt{t}%
\right] \text{.}
\end{eqnarray*}%
The polynomial $x-\theta \mu t$ is absorbed by exponential function in the
second inequality.
\end{proof}

From the Young's inequality for convolution, together with Lemma \ref{Lemm-H}%
, it follows that
\begin{align*}
\lvert h^{b}-h^{a}\rvert _{L_{x}^{r}}& \leq \lvert \frac{e^{-\frac{%
\left \vert x-\mu t\right \vert ^{2}}{4\kappa t}}}{(4\pi \kappa t)^{3/2}}-%
\frac{e^{-\frac{\left \vert x\right \vert ^{2}}{4t}}}{(4\pi t)^{3/2}}\rvert
_{L_{x}^{p}}\lvert h_{0}\rvert _{L_{x}^{q}} \\
& \leq Ct^{-\frac{3}{2}(1-\frac{1}{p})}\left[ \lvert \kappa -1\rvert +\lvert
\mu \rvert \sqrt{t}\right] \lvert h_{0}\rvert _{L_{x}^{q}} \\
& \leq Ct^{-\frac{3}{2}(\frac{1}{q}-\frac{1}{r})}\left[ \lvert \kappa
-1\rvert +\lvert \mu \rvert \sqrt{t}\right] \lvert h_{0}\rvert _{L_{x}^{q}}%
\text{,}
\end{align*}%
where $\frac{1}{p}+\frac{1}{q}=1+\frac{1}{r}$, $p$, $q$, $r\geq 1$.
Therefore, taking $q=1$, $r=\infty $ and $q=1$, $r=2$, respectively, gives
\begin{equation*}
\left \vert h^{b}-h^{a}\right \vert _{L_{x}^{\infty }}\leq C\left \vert
h_{0}\right \vert _{L_{x}^{1}}\left( \left \vert \kappa -1\right \vert
t^{-3/2}+\left \vert \mu \right \vert t^{-1}\right) \text{,}
\end{equation*}%
\begin{equation*}
\left \vert h^{b}-h^{a}\right \vert _{L_{x}^{2}}\leq C\left \vert
h_{0}\right \vert _{L_{x}^{1}}\left( \left \vert \kappa -1\right \vert
t^{-3/4}+\left \vert \mu \right \vert t^{-1/4}\right) \text{.}
\end{equation*}%
Noting that $\left \vert \kappa -1\right \vert \leq \left \vert \lambda
-1\right \vert $, the proof of Theorem \ref{heat} is completed.

Next, we provide an alternative proof for the difference. Let $h=h^{a}-h^{b}$%
. Then $h$ satisfies the equation%
\begin{equation*}
\left \{
\begin{array}{l}
\partial _{t}h=\Delta h+\mu \cdot \nabla h^{b}-\left( \kappa -1\right)
\Delta h^{b}\,,
\vspace {3mm}
\\
h\left( 0,x\right) =0\text{,}%
\end{array}%
\right.
\end{equation*}%
and it is given by
\begin{eqnarray*}
h\left( t,x\right) &=&\int_{0}^{t}\int_{\mathbb{R}^{3}}\frac{1}{\left[ 4\pi
\left( t-s\right) \right] ^{3/2}}e^{-\frac{\left \vert x-y\right \vert ^{2}}{%
4\left( t-s\right) }}\left[ \mu \cdot \nabla h^{b}\left( s,y\right) -\left(
\kappa -1\right) \Delta h^{b}\left( s,y\right) \right] dyds%
\vspace {3mm}
\\
&=&\int_{0}^{\frac{t}{2}}\int_{\mathbb{R}^{3}}\frac{1}{\left[ 4\pi \left(
t-s\right) \right] ^{3/2}}e^{-\frac{\left \vert x-y\right \vert ^{2}}{%
4\left( t-s\right) }}\left[ \mu \cdot \nabla h^{b}\left( s,y\right) -\left(
\kappa -1\right) \Delta h^{b}\left( s,y\right) \right] dyds \\
&&+\int_{\frac{t}{2}}^{t}\int_{\mathbb{R}^{3}}\frac{1}{\left[ 4\pi \left(
t-s\right) \right] ^{3/2}}e^{-\frac{\left \vert x-y\right \vert ^{2}}{%
4\left( t-s\right) }}\left[ \mu \cdot \nabla h^{b}\left( s,y\right) -\left(
\kappa -1\right) \Delta h^{b}\left( s,y\right) \right] dyds \\
&\equiv &h_{1}\left( t,x\right) +h_{2}\left( t,x\right) \text{.}\,
\end{eqnarray*}%
Recall the fact that
\begin{equation*}
\left \vert \partial _{x}^{\alpha }h^{b}\right \vert _{L_{x}^{q}}\leq
C\left( 1+t\right) ^{-\frac{3}{2}\left( 1-\frac{1}{q}\right) -\frac{\left
\vert \alpha \right \vert }{2}}\left \vert f_{0}\right \vert _{L_{x}^{1}}%
\text{, }1\leq q\leq \infty \text{,}
\end{equation*}%
for some constant $C>0$, 
where $\alpha =\left( \alpha _{1},\alpha _{2},\alpha _{3}\right) $ is a
multi-index. Thus, the integration by parts gives
\begin{eqnarray*}
h_{1}\left( t,x\right) &=&\int_{0}^{\frac{t}{2}}\int_{\mathbb{R}^{3}}\frac{1%
}{\left[ 4\pi \left( t-s\right) \right] ^{3/2}}e^{-\frac{\left \vert
x-y\right \vert ^{2}}{4\left( t-s\right) }}\left[ \mu \cdot \nabla
h^{b}\left( s,y\right) -\left( \kappa -1\right) \Delta h^{b}\left(
s,y\right) \right] dyds \\
&=&\int_{0}^{\frac{t}{2}}\int_{\mathbb{R}^{3}}\frac{1}{\left[ 4\pi \left(
t-s\right) \right] ^{3/2}}e^{-\frac{\left \vert x-y\right \vert ^{2}}{%
4\left( t-s\right) }}\frac{-\left( x-y\right) }{2\left( t-s\right) }\cdot %
\left[ \mu h^{b}\left( s,y\right) -\left( \kappa -1\right) \nabla
h^{b}\left( s,y\right) \right] dyds\text{,}
\end{eqnarray*}%
and for $t\geq 1$,%
\begin{eqnarray*}
\left \vert h_{1}\right \vert _{L_{x}^{r}} &\leq &C\left \vert f_{0}\right
\vert _{L_{x}^{1}}\int_{0}^{\frac{t}{2}}\left( t-s\right) ^{-\frac{3}{2}%
\left( 1-\frac{1}{p}\right) -\frac{1}{2}}\left( \left \vert \mu \right \vert
\left( 1+s\right) ^{-\frac{3}{2}\left( 1-\frac{1}{q}\right) }+\left \vert
\kappa -1\right \vert \left( 1+s\right) ^{-\frac{3}{2}\left( 1-\frac{1}{q}%
\right) -\frac{1}{2}}\right) ds \\
&\lesssim &\left \vert \mu \right \vert \left( 1+t\right) ^{-1+\frac{3}{2r}%
}+\left \vert \kappa -1\right \vert \left( 1+t\right) ^{-\frac{3}{2}+\frac{3%
}{2r}}
\end{eqnarray*}%
by the Young convolution inequality with $\frac{1}{p}+\frac{1}{q}=1+\frac{1}{%
r}$ and $-\frac{3}{2}\left( 1-\frac{1}{q}\right) +\frac{1}{2}>0$. With the
same $r$,
\begin{eqnarray*}
\left \vert h_{2}\right \vert _{L_{x}^{r}} &\leq &\int_{\frac{t}{2}%
}^{t}\left \vert \frac{1}{\left[ 4\pi \left( t-s\right) \right] ^{3/2}}e^{-%
\frac{\left \vert x\right \vert ^{2}}{4\left( t-s\right) }}\right \vert
_{L_{x}^{p}}\left \vert \mu \cdot \nabla h^{b}-\left( \kappa -1\right)
\Delta h^{b}\right \vert _{L_{x}^{q}}ds \\
&\leq &C\left \vert h_{0}\right \vert _{L_{x}^{1}}\int_{\frac{t}{2}%
}^{t}\left( t-s\right) ^{-\frac{3}{2}\left( 1-\frac{1}{p}\right) }\left(
\left \vert \mu \right \vert \left( 1+s\right) ^{-\frac{3}{2}\left( 1-\frac{1%
}{q}\right) -\frac{1}{2}}+\left \vert \kappa -1\right \vert \left(
1+s\right) ^{-\frac{3}{2}\left( 1-\frac{1}{q}\right) -1}\right) ds \\
&\lesssim &\left \vert \mu \right \vert \left( 1+t\right) ^{-1+\frac{3}{2r}%
}+\left \vert \kappa -1\right \vert \left( 1+t\right) ^{-\frac{3}{2}+\frac{3%
}{2r}}
\end{eqnarray*}%
for $t\geq 1$, by the Young convolution inequality with $\frac{1}{p}+\frac{1%
}{q}=1+\frac{1}{r}$, $-\frac{3}{2}\left( 1-\frac{1}{p}\right) +1>0$. Hence,
for $1\leq r\leq \infty $,%
\begin{equation*}
\left \vert h\right \vert _{L_{x}^{r}}\leq C\left( \left \vert \mu \right
\vert \left( 1+t\right) ^{-1+\frac{3}{2r}}+\left \vert \kappa -1\right \vert
\left( 1+t\right) ^{-\frac{3}{2}+\frac{3}{2r}}\right) \text{, }t\geq 1\text{,%
}
\end{equation*}%
where $C>0$ is a constant independent\ of $\lambda $, $u$ and $r$. In
particular,
\begin{equation*}
\left \vert h\right \vert _{L_{x}^{\infty }}\leq C\left( \left \vert \mu
\right \vert \left( 1+t\right) ^{-1}+\left \vert \kappa -1\right \vert
\left( 1+t\right) ^{-\frac{3}{2}}\right) \text{,}
\end{equation*}%
\begin{equation*}
\left \vert h\right \vert _{L_{x}^{2}}\leq C\left( \left \vert \mu \right
\vert \left( 1+t\right) ^{-\frac{1}{4}}+\left \vert \kappa -1\right \vert
\left( 1+t\right) ^{-\frac{3}{4}}\right) \text{,}
\end{equation*}%
for $t\geq 1$.

\textbf{Acknowledgments:} This work is partially supported by the National Key R\&D Program of
China, Project 2022YFA1000087. Y.-C. Lin is supported by the National Science and Technology
Council under the grant NSTC 110-2115-M-006-002-MY2. H.T. Wang is supported by NSFC
under Grant No. 12031013 and 12161141004, the Strategic Priority Research Program of Chinese
Academy of Sciences under Grant No. XDA25010403. K.-C. Wu is supported by the National
Science and Technology Council under the grant NSTC 112-2636-M-006 -001 and National Center
for Theoretical Sciences.
\newline
\newline

\end{document}